\documentclass[12pt, a4paper]{amsart}
\usepackage{graphicx}
\usepackage{amsmath, enumerate}
\usepackage{amsfonts, amssymb, amsthm, xcolor, stmaryrd}
\usepackage{tikz, tikz-cd, caption, tabu, longtable, mathrsfs, xtab}
 \usepackage{dsfont, cite, bm, amsbsy}
 \usepackage[colorinlistoftodos,prependcaption]{todonotes}
\usepackage{soul}
\usepackage[normalem]{ulem}
\usepackage[toc,page]{appendix}
\usepackage{hyperref}

\DeclareMathAlphabet{\mathpzc}{OT1}{pzc}{m}{it}

\numberwithin{equation}{subsection}
\theoremstyle{plain}

\newtheorem{thm}{Theorem}[section]
\newtheorem{lemma}[thm]{Lemma}
\newtheorem{prop}[thm]{Proposition}

\theoremstyle{definition}

\newtheorem{exmp}[thm]{Example}

\theoremstyle{remark}
\newtheorem{rmk}[thm]{Remark}

\newcommand{\Gal}{{\mathrm{Gal}}}
\newcommand{\GL}{{\mathrm{GL}}}

\newcommand{\Nm}{{\mathrm{Nm}}}

\newcommand{\SL}{{\mathrm{SL}}}
\newcommand{\SO}{{\mathrm{SO}}}

\newcommand{\ord}{{\mathrm{ord}}}

\newcommand{\tth}{\textsuperscript{th }}
\newcommand{\sgn}{\mathrm{sgn}}

\newcommand{\Tr}{\mathrm{Tr}}

\newcommand{\Cl}{{\mathrm{Cl}}}

\newcommand{\Cb}{\mathbb{C}}

\newcommand{\Nb}{\mathbb{N}}
\newcommand{\Qb}{\mathbb{Q}}
\newcommand{\Qbar}{\overline{\mathbb{Q}}}
\newcommand{\Rb}{\mathbb{R}}
\newcommand{\Zb}{\mathbb{Z}}

\newcommand{\ebf}{\mathbf{e}}

\newcommand{\af}{\mathfrak{a}}

\newcommand{\bfrak}{\mathfrak{b}}
\newcommand{\df}{\mathfrak{d}}

\newcommand{\ef}{\mathfrak{e}}

\newcommand{\mf}{\mathfrak{m}}

\newcommand{\Ac}{{\mathcal{A}}}

\newcommand{\Cc}{{\mathcal{C}}}
\newcommand{\Dc}{{\mathcal{D}}}
\newcommand{\Ec}{{\mathcal{E}}}
\newcommand{\Fc}{{\mathcal{F}}}

\newcommand{\Hc}{{\mathcal{H}}}
\newcommand{\Ic}{{\mathcal{I}}}

\newcommand{\Oc}{\mathcal{O}}

\newcommand{\Sc}{\mathcal{S}}

\newcommand{\half}{{\tfrac{1}{2}}}

\renewcommand{\Re}{\mathrm{Re}}
\newcommand{\varep}{\varepsilon}

\newcommand{\lp}{\left (}
\newcommand{\rp}{\right )}

\newcommand{\erf}{\mathrm{erf}}
\newcommand{\erfc}{\mathrm{erfc}}

\newcommand{\gh}{g^+}
\newcommand{\gn}{g^*}

\newcommand{\Sha}{\mathcal{C}}

\newcommand{\tTheta}{\tilde{\Theta}}
\newcommand{\tvartheta}{\tilde{\vartheta}}
\newcommand{\varphif}{\varphi_{\mathrm{f}}}

\newcommand{\mbf}{{\mathbf{m}}}
\newcommand{\cbf}{{\textit{\textbf{c}}}}
\newcommand{\fbf}{{\textit{\textbf{f}}}}

\newcommand{\tphi}{{\tilde{\phi}}}

\newcommand{\tf}{{\tilde{f}}}

\newcommand{\VR}{{V_\mathbb{R}}}

\newcommand{\blam}{\boldsymbol{\lambda}}
\newcommand{\thetab}{\boldsymbol{\theta}}
\newcommand{\smat}[4]{\left(\begin{smallmatrix}
                 #1 & #2\\
                 #3 & #4
\end{smallmatrix}\right)}

\newcommand{\pmat}[4]{\begin{pmatrix}
                 #1 & #2\\
                 #3 & #4
\end{pmatrix}}

\newcommand{\pars}[1]{(\!(#1 )\!)}

\makeatletter
\providecommand\@dotsep{5}
\renewcommand{\listoftodos}[1][\@todonotes@todolistname]{%
  \@starttoc{tdo}{#1}}
\makeatother

\begin{document}
\title{Harmonic Maass Forms Associated to Real Quadratic Fields}

\author[Pierre Charollois and Yingkun Li]{Pierre Charollois and Yingkun Li}
\address{Equipe de th\'eorie des nombres, Institut de Math\'ematiques de Jussieu-PRG, Case 247. 4, Place Jussieu, 75252 Paris Cedex, France}
\email{pierre.charollois@imj-prg.fr}
\address{Fachbereich Mathematik,
Technische Universit\"at Darmstadt, Schlossgartenstrasse 7, D--64289
Darmstadt, Germany}
\email{li@mathematik.tu-darmstadt.de}
\thanks{The first author is partially supported by the  ANR grant  ANR-12-BS01-0002.\\
\indent The second author is partially supported by the DFG grant BR-2163/4 and an NSF postdoctoral fellowship.}

\begin{abstract}
 In this paper, we explicitly construct harmonic Maass forms that map to the holomorphic weight one theta series associated by Hecke to odd ray class group characters of real quadratic fields. From this construction, we give precise arithmetic information contained in the Fourier coefficients of the holomorphic part of the harmonic Maass form, establishing the main part of a conjecture of the second author.

\end{abstract}

\date{\today}

\maketitle

\tableofcontents
\section{Introduction.}
\label{sec:Intro}

In number theory, modular forms of weight one play an important role because of their special relationship to number fields. 
%
To each weight one eigenform $f \in S_{1, \chi}(\Gamma_0(N))$, Deligne and Serre functorially attached a 2-dimensional, odd, irreducible representation $\varrho_f$ of $\Gal(\Qbar/\Qb)$ \cite{DS74}.
This representation gives rise to a finite Galois extension $M/\Qb$ and it is natural to expect $f$ to encode interesting arithmetic information about $M$.
In fact, Stark's conjecture \cite{Stark77} predicts that a certain subfield of $M$ can be constructed from the special values of the $L$-function attached to the modular form $f$.

In \cite{Hecke26}, Hecke gave several systematic constructions of weight one modular forms.
One of them attached weight one cusp forms, whose Galois representations have dihedral projective images, to real quadratic fields. 
Specifically, let $F \subset \Rb$ be a real quadratic field with discriminant $D > 0$ and fundamental unit $\varepsilon_F > 1$. 
Let $\mf$ be an integral ideal in $F$ with norm $M$ and $\varphi$ an odd ray class group character with conductor $\mf \cdot \infty_1$. 
Then one can associate an eigenform 
$$
f_\varphi(\tau) = \sum_{n \ge 1} c_\varphi(n) q^n := \sum_{\af \subset \Oc_F} \varphi(\af) q^{\Nm(\af)} \in S_{1, \chi}(\Gamma_0(DM)),
$$
where $\tau$ is in the upper half complex plane $\Hc$, $q := e^{2\pi i \tau}$, and $\chi = \chi_D \cdot  \varphi|_{\Qb}$ (see Section \ref{subsec:ray} for details).
The Galois representation associated to $f_\varphi$ by Deligne and Serre is the induction of $\varphi$ from $\Gal(\overline{\Qb}/F)$ to $\Gal(\overline{\Qb}/\Qb)$. 
Hecke's construction was originally in terms of vector-valued modular forms and one of the first instances of producing holomorphic modular forms using a theta lift for indefinite quadratic forms. This work had been vastly generalized by Kudla and Millson \cite{Kudla81, KM90}.

It turns out the method of theta liftings has far reaching consequences in number theory and arithmetic geometry.
In \cite{Borcherds98}, Borcherds constructed automorphic forms with singularities on orthogonal Shimura varieties via a regularized theta lift. The singularities of the output are controlled by the input, which is a holomorphic modular form with poles at the cusps. Then in \cite{Bruinier99} and \cite{Bruinier02}, Bruinier replaced this input with certain non-holomorphic Poincar\'{e} series, which are eigenfunctions of the hyperbolic Laplacian. He used the outputs to produce Chern classes for the Heegner divisors.
%
Motivated by these works, Bruinier and Funke introduced in  \cite{BF04} the notion of \textit{harmonic Maass forms} generalizing classical modular forms.  
They are annihilated by the weight-$k$ hyperbolic Laplacian
\begin{equation}
  \label{eq:Deltak}
 \Delta_k := \xi_{2-k} \circ \xi_k, \; \xi_k := 2 i v^{k} \overline{\frac{\partial}{\partial \overline{\tau}}}, \; \tau = u + iv \in \Hc
\end{equation}
 and have polar type singularities at the cusps (see Section \ref{subsec:weilrep}). 
Using harmonic Maass forms as the input to Borcherds' regularized theta lift, Bruinier and Funke constructed an adjoint of the Kudla-Millson theta lift for orthogonal groups of arbitrary signature \cite{BF04}.
This theta lift then produces automorphic Green's function for special divisors on orthogonal type Shimura varieties, which enables one to calculate arithmetic intersection numbers and leads to generalizations of the famous Gross-Zagier formula \cite{BY09} and the recent proof of an averaged version of Colmez's conjecture \cite{AGHM2}.

Besides as input to the theta lift, harmonic Maass forms also have interesting Fourier coefficients. 
Because of the annihilation by $\Delta_k$, a harmonic Maass form naturally has a holomorphic part and a non-holomorphic part in its Fourier expansion.
In his ground breaking thesis \cite{ZwThesis}, Zwegers completed Ramanujan's holomorphic mock theta functions and produced real-analytic modular forms of weight $\half$. They turned out to be harmonic Maass forms that map to weight $\frac{3}{2}$ unary theta series under $\xi_{1/2}$. 
In weight $\half$, there are many other important works that reveal the arithmetic nature of these Fourier coefficients (see e.g. \cite{BO10}, \cite{DIT11}).

In the self-dual case of weight $k = 1$, Kudla, Rapoport and  Yang \cite{KRY99} constructed ``incoherent Eisenstein series'', which turned out to be harmonic Maass forms that map under $\xi_{1}$ to Eisenstein series associated to an imaginary quadratic field $K$. The Fourier coefficients of the holomorphic part are logarithms of integers, and can be interpreted as arithmetic degrees of special divisors on arithmetic curves. 
Later Duke and the second author studied in \cite{DL15} harmonic Maass forms that map   to weight one cusp forms associated to non-trivial class group characters of $K$.
The Fourier coefficients were shown to be logarithms of algebraic numbers in the Hilbert class field of $K$. 
In his thesis \cite{Ehlen_thesis}, Ehlen gave an arithmetic interpretation of the valuation of these algebraic numbers along the lines of \cite{KRY99}.
In contrast to the incoherent Eisenstein series in \cite{KRY99}, the harmonic Maass forms in \cite{DL15} and \cite{Ehlen_thesis} were not constructed explicitly. 
Also, numerical evidence suggests that given any weight one eigenform $f$ with associated Galois representation $\varrho$, there exists a harmonic Maass form $\tilde{f}$ such that $\xi_1 \tilde{f} = f$ and the Fourier coefficients of its holomorphic part are $\Qbar$-linear combinations of logarithms of algebraic numbers in the number field cut out by $\mathrm{ad} \varrho$ (see \cite{DL15}, \cite{Li16}).

In this paper, we will \textit{explicitly construct} a harmonic Maass form $\tilde{f}_\varphi$ that maps under $\xi_{1}$  to Hecke's weight one cusp form $f_\varphi$, and study the arithmetic information contained in its Fourier coefficients.
One special case of our main result (Theorems \ref{thm:Main} and \ref{thm:Main_sc}) is as follows.

\begin{thm}
  \label{thm:Main_sc_1}
Let $\varphi$ be an odd ray class group character of conductor $\mf \cdot \infty_1$ with $\Nm(\mf) = M$ and $f_\varphi$ the holomorphic weight one eigenform associated to $\varphi$ as above.
Suppose $F$ has class number 1. Then there exists an integer $\kappa_\mf$ dividing $48 M^3 \phi(2M)$ and a harmonic Maass form $\tilde{f}_\varphi \in H_{1, \overline{\chi}}(\Gamma_0(DM))$ with holomorphic part $\tilde{f}^+_\varphi(\tau) = \sum_{n \gg -\infty} c^+_\varphi(n) q^n$ such that $\xi_1 \tilde{f}_\varphi = f_\varphi$ and 
\begin{equation}
  \label{eq:c+_1}
  c^+_\varphi(n) 
- \frac{1}{2} \sum_{(\lambda) \subset \Oc_F, \; \Nm((\lambda)) = n } \lp {\varphi^{-1}}(\lambda)   - {\varphi^{-1}}(\lambda') \rp\log \left| \frac{\lambda}{\lambda'} \right| \in 
\frac{1}{\kappa_\mf} \Zb[\varphi] \cdot \log \varep_F,
\end{equation}
where $\phi(N) := [\SL_2(\Zb): \Gamma_0(N)]$ and $\Zb[\varphi] \subset \Cb$ is the subring generated by the values of $\varphi$.
\end{thm}

\begin{rmk}
The class number assumption is only to ease exposition. The more general result is in Theorem \ref{thm:Main_sc}.
Since Hecke's construction is naturally stated in the setting of vector-valued modular forms transforming with respect to the Weil representation, we will first obtain the main result in this setting (Theorem \ref{thm:Main}) and deduce Theorem \ref{thm:Main_sc} from it.
\end{rmk}

\begin{rmk}
  For each $n \ge 1$, we can define an analogue of $c_\varphi(n)$ by
\begin{equation}
  \label{eq:cbf}
  \cbf_\varphi(n) := 
\frac{1}{2} \sum_{(\lambda) \subset \Oc_F, \; \Nm((\lambda)) = n } \lp {\varphi^{-1}}(\lambda)   - {\varphi^{-1}}(\lambda') \rp\log \left| \frac{\lambda}{\lambda'} \right| \in \Cb /R_\varphi
\end{equation}
with $R_\varphi := \Zb[\varphi] \cdot \log \varep_F$ and form the formal power series
\begin{equation}
  \label{eq:fbf}
\fbf_\varphi := \sum_{n \ge 1} \cbf_\varphi(n) q^n \in
\Cb/R_\varphi\llbracket q \rrbracket.
\end{equation}
Then the second author conjectured in \cite{Li16} that $\fbf_\varphi$ can be lifted to the holomorphic part of a harmonic Maass form mapping to $f_\varphi$ under $\xi_1$. 
Now Theorem \ref{thm:Main_sc_1} implies a slightly weaker result that $\kappa_\mf \fbf_\varphi$ can be lifted, i.e.\ $\kappa_\mf \fbf_\varphi$ agrees with $\kappa_\mf \tilde{f}_\varphi^+$ as formal power series in $\Cb/R_\varphi \llbracket q \rrbracket$.
In specific cases, it is possible to reduce the bound $48 M^3\phi(2M)$ above with careful analysis.
Note that if $n = \ell$ or $\ell^2$ with $\ell$ an inert prime in $F/\Qb$, then Theorem \ref{thm:Main_sc_1} implies that $\kappa_\mf c^+_\varphi(n) \in R_\varphi$. 
\end{rmk}

\begin{rmk}
In \cite{DLR15a}, Darmon, Lauder and Rotger studied a non-classical, overconvergent generalized eigenform associated to $f_\varphi$. In their $p$-adic setting, the $\ell$-th Fourier coefficient is the  $p$-adic logarithm of a Gross-Stark $\ell$-unit in a class field of $F$ when $\ell$ is an inert prime in $F/\Qb$. Otherwise, it is zero. In Section \ref{subsec:ray}, we will rewrite $\cbf_\varphi(n)$ to show that the archimedean and non-archimedean settings are in some sense complementary to each other. 
\end{rmk}

\begin{rmk}
In \cite{Li16}, it is shown that certain $\Zb[\varphi]$-linear combinations of the $c^+_\varphi(n)$'s are the values of Hilbert modular functions at big CM points. Thus, getting a handle on the individual $c^+_\varphi(n)$ allows us to give explicit factorization formula of the CM values in the spirit of Gross and Zagier \cite{GZ85}. Furthermore the CM values are defined over $F$. We hope to pursue this line of investigation in the future.
\end{rmk}


In order to best reflect the nature of the coefficients $c^+_\varphi(n)$, we have stated Theorem \ref{thm:Main_sc_1} as an existence result. 
Its proof is through \textit{explicit construction} and it is possible to write down a closed formula for $c^+_\varphi(n)$ as a finite sum of $\Qb$-linear combination of logarithms of algebraic numbers in $F$. We will do this numerically for $\eta^2(\tau)$ at the end.

 The construction starts by deforming Hecke's theta integral with a spectral parameter $s \in \Cb$ (see \eqref{eq:s_int}).
The derivative of this deformed integral at $s = 0$ differs from the desired harmonic Maass form by a non-modular contribution $-\tTheta^*(\tau, L)$ from the boundary. We then construct a modular object $\tTheta(\tau, L)$ that differs from this boundary contribution by a translation-invariant holomorphic function $\tTheta^+(\tau, L)$, whose Fourier coefficients are in $\frac{1}{\kappa_\mf} \Zb[\varphi] \cdot \log \varep_F$. 
Their difference is then the desired harmonic Maass form.

The outline of the paper is as follows. In Section \ref{sec:lift}, we will recall some basic information about vector-valued automorphic forms and theta functions following \cite{Kudla81} and \cite{Borcherds98}. 
In Section \ref{sec:special_func}, we will define a continuous function $g_\tau \in L^1(\Rb) \cap L^\infty(\Rb)$ with nice properties (see Proposition \ref{prop:gtau}) and use it in Section \ref{sec:tTheta} to construct $\tTheta(\tau, L)$, which offsets the effect of the boundary contribution to the deformed theta integral.
Finally in the last two sections, we will state and prove the vector-valued and scalar-valued versions of our main result, before giving some numerical examples at the end.

\section*{Acknowledgment}
The authors thank Jan Bruinier for many illuminating conversations, and the anonymous referee for detailed comments and several helpful suggestions, which greatly improved the exposition of the paper.

\section{Theta Lift from $\mathrm{O}(1, 1)$ to $\SL_2$.}
\label{sec:lift}

In this section, we will follow Kudla \cite{Kudla81} to recall the work of Hecke \cite{Hecke26} in terms of a theta lift from $\mathrm{O}(1, 1)$ to $\SL_2$. 

\subsection{Indefinite, anisotropic $\Zb$-lattice of rank 2.}
\label{subsec:lattice}
An even, integral lattice is a $\Zb$-module $L$ equipped with a quadratic form $Q: L \to \Zb$. 
It is anisotropic if $L$ does not contain an nonzero isotropic vector, i.e. $\lambda \in L \backslash\{0\}$ with $Q(\lambda) = 0$. 
If such lattice is indefinite and has rank 2, then it can be described by $L = \Zb^2$ and a quadratic form $Q(m, n) = Am^2 + B mn + C n^2$ with $A, B, C \in \Zb$ such that the discriminant $D:= B^2 - 4AC > 0$ is not a perfect square.
Furthermore, if $D$ is fundamental, then $A\Zb + \frac{B + \sqrt{D}}{2} \Zb$ is an integral ideal in the real quadratic field $\Qb(\sqrt{D})$ with norm $|A| \neq 0$ and
$$
\Nm\lp  A m + \frac{B+ \sqrt{D}}{2} n \rp =  A \cdot Q(m, n).
$$
In fact, there is a correspondence between indefinite binary quadratic forms and ideals in real quadratic fields (see \cite{Buell89}). We will use the latter language throughout this work.

Let $\Qbar$ be an algebraic closure of $\Qb$ and fix an embedding $\Qbar \hookrightarrow \Cb$ throughout.
For a discriminant $D \ge 1$ (not necessarily fundamental), let $F = \Qb(\sqrt{D})$ be the corresponding real quadratic field with ring of integers $\Oc_F$.
The $\Zb$-lattice $\Oc_D := \Zb + \Zb \tfrac{D + \sqrt{D}}{2}$ is a subring of $\Oc_F$. Its dual with respect to the norm form $\Nm$ as the quadratic form is $\df_D^{-1} \Oc_F$, where $ \df_D := \sqrt{D} \Oc_D$.
If $D$ is fundamental, then $\Oc_D = \Oc_F$ and $\df_D = \df_F$ is the different.
The group of units $\Oc_F^\times$ is generated by $- 1$ and the fundamental unit $\varepsilon_F > 1$.

For an integral ideal $\af \subset \Oc_D$ with $A := [\Oc_D: \af]$ and a positive integer $M \in \Nb$, consider the $\Zb$-lattice
\begin{equation}
  \label{eq:Lam}
(  L_{\af, M} , Q_{\af, M}) := \lp M \af, \frac{\Nm_{F/\Qb}}{AM} \rp.
\end{equation}
It is anisotropic and has rank 2.
The induced bilinear form is given by
$
B_{\af, M}(\lambda, \mu) := \frac{\Tr_{F/\Qb}(\lambda \cdot \mu')}{AM}
$
for all $\lambda, \mu \in L_{\af, M}$ with $'$ the non-trivial automorphism in $\Gal(F/\Qb)$. 
The dual lattice $L^*_{\af, M}$ is given by $\af \df_D^{-1}$ and the finite quadratic modular $L^*_{\af, M}/L_{\af, M}$ is isomorphic to $\Oc_D/M\df_D$.
Finally, notice that $-L_{\af, M}$ is isometric to $L_{\af \df_D, M}$ via $\lambda \mapsto \lambda \sqrt{D}$. 


\subsection{Weil representation and automorphic forms.}
\label{subsec:weilrep}

Let $(L, Q)$ be an indefinite, even, integral lattice of rank 2 with bilinear form $(,):L\times L \to \Zb$.
As usual, let $L^*$ be the dual lattice of $L$ and $\{\ef_h: h \in L^*/L \}$ denote the canonical basis of the vector space $\Cb[L^*/L]$ and $\ebf(a) := e^{2\pi i a}$ for any $a \in \Cb$.
Then $\Gamma:= \SL_2(\Zb)$ acts on $\Cb[L^*/L]$ through the Weil representation $\rho_L$ as (see e.g.\ \cite[\textsection 4]{Borcherds98})
\begin{equation}
\label{eq:Weil_rep}
  \rho_L(T) (\ef_h) = \ebf(Q(h)) \ef_h,  \;
\rho_L(S) (\ef_h) = \frac{1}{\sqrt{|L^*/L|}} \sum_{\delta \in L^*/L} \ebf(-(\delta, h)) \ef_\delta,
\end{equation}
where $T=\left(\begin{array}{cc}1&1\\ 0&1\end{array}\right),S=\left(\begin{array}{cc}0&-1\\1&0\end{array}\right).$ 
Note that $\rho_{-L} = c \circ {\rho_L} \circ c$ on $\Cb[L^*/L]$, where $c: \Cb \to \Cb$ denotes complex conjugation.
If $P \subset L$ is a finite index sublattice, then $L^* \subset P^*$ and denote $s: L^*/P \to L^*/L$ the natural surjection.
There is a linear map $\psi: \Cb[P^*/P] \to \Cb[L^*/L]$ defined by 
\begin{equation}
  \label{eq:psi}
  \psi(\ef_h) :=
  \begin{cases}
    \ef_{s(h)}, & h \in L^*/P\\
0, & \text{ otherwise.}
  \end{cases}
\end{equation}
It is straightforward to check that $\psi$ is in fact $\Gamma$-linear (with respect to $\rho_P$ and $\rho_L$).

Let $d_L$ be the level of $L$ and $\Gamma(d_L) \subset \Gamma$ the principal congruence subgroup of level $d_L$. Then $\rho_L$ is trivial on $\Gamma(d_L)$ and can be viewed as a representation of $\SL_2(\Zb/d_L \Zb)$ (see e.g.\ Proposition 4.5 in \cite{Sch09}).
Let $\zeta_{d_L}$ be a primitive $d^{\mathrm{th}}_L$ root of unity. 
For $a \in (\Zb/d_L \Zb)^\times$, let $\sigma_a \in \Gal(\Qb(\zeta_{d_L})/\Qb)$ be the element that sends $\zeta_{d_L}$ to $\zeta^a_{d_L}$. 
Then $\sigma_a$ acts naturally on $W_L := \Qb(\zeta_{d_L})[L^*/L]$.
Let $\varsigma_a \in \GL(W_L)$ be the left action given by
\begin{equation}
  \label{eq:varsigma}
  \varsigma_a \cdot w := \sigma_a^{-1}(w), \; w \in W_L.
\end{equation}
In \cite{Mc03}, McGraw extended $\rho_L$ to a unitary representation of $\GL_2(\Zb/d_L \Zb)$ on $W_L$, where the action of 
\begin{equation}
  \label{eq:Ja}
J_a := \smat{1}{}{}{a} \in \GL_2(\Zb/d_L \Zb)  
\end{equation}
is $\varsigma_a$. This is the main ingredient used to prove the rationality of basis of vector-valued modular forms. The result we need can be stated as follows.
\begin{prop}\cite[Theorem 4.3]{Mc03}
  \label{prop:rationality}
The map that sends $\gamma$ to $\rho_L(\gamma)$ when $\gamma \in \SL_2(\Zb/d_L\Zb)$ and $J_a$ to $\varsigma_a$ is a unitary representation of $\GL_2(\Zb/d_L \Zb)$ on $W_L$.
In other words, if we view $\rho_L(\gamma) \in \GL_{|L^*/L|}(\Qb(\zeta_{d_L}))$ with respect to the standard basis of $W_L$, then $\sigma_a(\rho_L(\gamma)) = \rho_L(J_a^{-1} \gamma J_a)$.
\end{prop}

Now, we will quickly recall some facts about automorphic forms. 
Let $k \in \Zb$ be an integer, $V$ an $n$-dimensional $\Cb$-vector space, $\Gamma' \subset \Gamma$ a finite index subgroup and $\rho: \Gamma' \to \GL(V)$ a representation. 
Then a real-analytic function $f = (f_j)_{1 \le j \le n} : \Hc \to V$ is a vector-valued automorphic form on $\Gamma'$ with weight $k$ and representation $\rho$ if it satisfies
\begin{equation}
  \label{eq:modularity}
(  f \mid_{k, \rho} \gamma)(\tau) :=  \rho(\gamma)^{-1} \cdot \lp  (c \tau + d)^{-k} f_j \lp \frac{a\tau + b}{c\tau + d} \rp \rp_{1\le j \le n} = f(\tau)
\end{equation}
for all $\gamma = \smat{a}{b}{c}{d} \in \Gamma'$ and $\tau \in \Hc$. 
If $\rho$ is trivial, then we may omit it in the slash operator.
We denote the space of such functions by $\Ac_{k, \rho}(\Gamma')$. 
The subspace of $\Ac_{k, \rho}(\Gamma')$ consisting of functions holomorphic on $\Hc$ is denoted by $M^!_{k, \rho}(\Gamma')$, which is usually called the space of weakly holomorphic modular forms.
Let $M_{k, \rho}(\Gamma')$ and $S_{k, \rho}(\Gamma')$ be the usual space of modular forms and cusp forms respectively.
A function $f \in \Ac_{k, \rho}(\Gamma')$ is called a harmonic weak Maass form if $\Delta_k f = 0$ and $f$ has at most linear exponential growths at all the cusps. 
It satisfies $\xi_k f \in M^!_{2-k, \overline{\rho}}(\Gamma')$ by the definition of $\Delta_k$ in \eqref{eq:Deltak}.
Furthermore, if $\xi_k f$ vanishes at all the cusps, then we call $f$ a harmonic Maass form.
We use $H_{k, \rho}(\Gamma')$ to denote the space of harmonic Maass forms on $\Gamma'$ of weight $k$ and representation $\rho$.

Since $f \in H_{k, \rho}(\Gamma')$ satisfies $\Delta_k f = 0$, it can be written canonically as the difference of a holomorphic part $f^+$ and non-holomorphic part $f^*$. Let $\mathbf{B}$ be a basis of $V$. Then $f^+$ and $f^*$ have the following Fourier expansions.
\begin{align*}
  f^+(\tau) &= \sum_{\ef \in \mathbf{B}} \lp \sum_{\begin{subarray}{c} n \in \Qb \\ n \gg -\infty  \end{subarray}} c^+(n, \ef) q^n  \rp \ef, \;
f^*(\tau) =  (4\pi)^{k-1} \sum_{\ef \in \mathbf{B}} \lp \sum_{\begin{subarray}{c} n \in \Qb \\ n > 0  \end{subarray}} c(n, \ef) \Gamma(1-k, 4\pi n v) q^{-n}  \rp \ef.
\end{align*}
It is readily checked that $\xi_k f(\tau) = \xi_k (- f^*(\tau)) = \sum_{\ef \in \mathbf{B}} \lp \sum_{\begin{subarray}{c} n \in \Qb \\ n > 0  \end{subarray}} n^{1-k} c(n, \ef) q^{n}  \rp \ef \in S_{2-k, \overline{\rho}}(\Gamma')$.

\subsection{Vector-valued theta functions.}
\label{subsec:vvTheta}
Let $(\VR, Q)$ denote the quadratic space of signature $(1, 1)$ with $\VR = \Rb^2$ and for $X = (x_1, x_2), Y = (y_1, y_2) \in \VR$
\begin{equation}
  \label{eq:Q}
Q(X) := x_1x_2, \; B(X, Y) := x_1 y_2 + x_2 y_1.
\end{equation}
Note that $(\VR, Q) \cong (\VR, -Q)$ via the map 
\begin{equation}
  \label{eq:iota}
 \iota ((x_1, x_2)) := (x_1, -x_2). 
\end{equation}
The symmetric domain attached to $\VR$ is the hyperbola 
$\Dc := \left\{Z^- \subset \VR | (Z^-, Z^-) = -1 \right\}$, which is parametrized by $\Rb^\times$ via
\begin{align*}
  \Phi: \Rb^\times &\to \Dc \\
t &\mapsto Z^-_t :=    \lp \tfrac{t}{\sqrt{2}}, -\tfrac{t^{-1}}{\sqrt{2}} \rp.
\end{align*}
We denote by $\Dc^+$ the connected component of $\Dc$, which is parametrized by $\Rb^\times_+$ under $\Phi$.
Define
\begin{equation}
  \label{eq:Z+}
Z^+_t := \frac{1}{\sqrt{2}} \lp t, t^{-1} \rp \in (Z^-_t)^\perp.  
\end{equation}
Then $d \Phi \lp t \tfrac{d}{dt} \rp = Z^+_t \in \Rb^2$ and $\{Z^+_t, Z^-_t\}$ is an orthogonal basis of $\VR$ with $Q(Z^+_t) = -Q(Z^-_t) = \tfrac{1}{2}$. We can write $X = X^+_t - X^-_t$ with
\begin{equation}
\label{eq:projection}
X^\alpha_t := B(X, Z^\alpha_t) Z^\alpha_t =  \frac{\alpha x_1 t^{-1} + x_2 t}{\sqrt{2}}Z^\alpha_t
\end{equation}
for any $X = (x_1, x_2) \in \VR$ and $\alpha \in \{+, -\}$.
In this basis, the quadratic space $\VR$ becomes $(\Rb^{1, 1}, Q_0)$ with $\Rb^{1, 1} = \Rb^2$ and 
\begin{equation}
  \label{eq:Q0}
Q_0((x, y)) := \frac{x^2 - y^2}{2}.
\end{equation}
Let $\Fc_0$ denote the Fourier transform on $\Rb^{1, 1}$ with respect to $Q_0$, i.e.
$$
\Fc_0(\phi)(x, y) := \int^{\infty}_{-\infty} \int^{\infty}_{-\infty} \phi(w, r) \ebf(xw - yr) dw dr
$$
for any Schwartz function $\phi$ on $\Rb^{1, 1}$.

An important ingredient in forming the theta kernel is the archimedean part of the Schwartz function. In the setting of Hecke, this is given by
\begin{equation}
\label{eq:phi}
\phi_\tau(x, y) := \sqrt{2v} \cdot x \cdot \ebf \lp \frac{x^2}{2} \tau - \frac{y^2}{2} \overline{\tau} \rp.
\end{equation}
As a function on the quadratic space $(\Rb^{1, 1}, Q_0)$, $\phi_\tau$ satisfies
\begin{equation}
  \label{eq:phi_prop}
  \phi_{\tau+1}(W) = \ebf \lp Q_0(W) \rp \phi_\tau(W), \;
\phi_{-1/\tau}(W) = -\tau \Fc_0(\phi_{\tau})(W), \; W \in \Rb^{1, 1}.
\end{equation}
Now for any even, integral lattice $L \subset \VR$, the vector-valued theta function 
$\Theta(\tau,  L; t) := \sum_{h \in L^*/L}  \Theta_{h}(\tau,  L; t) \ef_h$ with
\begin{equation}
  \label{eq:Theta}
\Theta_{h}(\tau, L; t) := \sum_{X \in L + h}
\phi_{\tau}(B(X, Z^+_t), B(X, Z^{-}_t))
\end{equation}
transforms on $\SL_2(\Zb)$ with weight 1 and representation $\rho_{L}$ in the variable $\tau$ by Theorem 4.1 in \cite{Borcherds98}.
Similarly, the image of $-L$ under the involution $\iota: -\VR \to \VR$ is an even, integral lattice in $(\VR, Q)$, and we define
\begin{equation}
  \label{eq:Theta-}
  \Theta(\tau, -L; t)   :=   \Theta(\tau, \iota(-L); t).
\end{equation}

Suppose $L$ is the image of the embedding
\begin{equation}
  \label{eq:embedding}
  \begin{split}
       (L_{\af, M},  Q_{\af, M}) &\hookrightarrow (\VR, Q) \\
\lambda &\mapsto \blam :=  \frac{1}{\sqrt{AM}} (\lambda,  \lambda').
  \end{split}
\end{equation}
for some $D, \af, M$. 
Then $(L, Q) \cong (L_{\af, M},  Q_{\af, M})$ and $\Theta_{h}(\tau, L; t)$ becomes
\begin{equation}
\label{eq:theta_explicit}
\Theta_{h} (\tau, L; t) = 
\frac{\sqrt{v}}{\sqrt{AM}} 
\sum_{
  \begin{subarray}{c}
    \lambda \in \af \df_D^{-1} \\
\lambda - h \in M \af 
  \end{subarray}}
\lp
\lambda't  + \lambda t^{-1}
\rp
\ebf\lp
\frac{
(\lambda t^{-1} + \lambda' t)^2 \tau - (\lambda t^{-1} -  \lambda't)^2 \overline{\tau}
}{4AM}
\rp.
\end{equation}
Similarly, the expression above becomes $\Theta_h(\tau, -L; t)$ after changing $\lambda'$ to $-\lambda'$. 
If $P \subset L$ is a sublattice of finite index and $\psi: \Cb[P^*/P] \to \Cb[L^*/L]$ the $\Gamma$-linear map defined in \eqref{eq:psi}, then it is straightforward to check that 
\begin{equation}
  \label{eq:equivariance}
  \psi(\Theta(\tau, P; t)) = \Theta(\tau, L; t)
\end{equation}
for all $\tau \in \Hc$ and $t \in \Rb^\times_+$. 


\subsection{Theta integral.}
As in the end of last subsection, suppose $(L, Q) = (L_{\af, M}, Q_{\af, M})$. 
The discriminant kernel $\Gamma_L$ is the subgroup of $\SO^+(L) \cong \Zb$ consisting of those units which are congruent to 1 modulo $M\sqrt{D}$. 
Notice that $\varep \in \Gamma_L$ if and only if $\varep' \in \Gamma_L$.
Also, $\Nm(\varep) = 1$ for all $\varep \in \Gamma_L$. 
Let $\Gamma'_L \subset \Gamma_L$ be subgroup of totally positive element and $\varep_L > 1$ its generator.
In \cite{Hecke26}, Hecke calculated the integral
\begin{equation}
\label{eq:vartheta}
\vartheta(\tau, L)
:=  \int^{ \varep_L}_{ 1} \Theta(\tau, L; t) \frac{dt}{t}
= \int^{ \log \varep_L}_{ 0} \Theta(\tau, L; e^{\nu}) d\nu.
\end{equation}

Since $B(\lambda \varep,  Z^\pm_t ) = B(\lambda,  Z^\pm_{\varep' t} )$ for any totally positive $\varep \in \Oc_F^\times$, we can unfold the integral to obtain
\begin{align*}
\int^{\varep_L}_1 \Theta_{h}(\tau, L; t) \frac{dt}{t} &= 
\sum_{\lambda \in \Gamma'_L \backslash  L + h, \; \lambda \neq 0}
\int^\infty_0
 \phi_\tau(B(\lambda, Z^+_t), B(\lambda, Z^-_t)) \frac{dt}{t}\\
&= 
\frac{\sqrt{v}}{\sqrt{AM}} 
\sum_{\lambda \in \Gamma'_L \backslash L + h, \; \lambda \neq 0} 
\ebf\lp
\frac{
\lambda \lambda' u
}{AM}
\rp
\int^\infty_{-\infty}
\lp \lambda' e^\nu  + \lambda e^{-\nu} \rp
\ebf\lp
\frac{
(\lambda e^{-\nu})^2 + (\lambda' e^\nu)^2 
}{2AM}iv
\rp
d\nu.
\end{align*}
Using the identity $e^{-2\pi y} = 2 \sqrt{y} \int^\infty_{-\infty} e^{\nu} e^{-\pi y (e^{2\nu} + e^{-2\nu})} d\nu$, we can evaluate
\begin{align}
\label{eq:integral_id}
\int^\infty_{-\infty}
\lambda' e^\nu
\ebf\lp
\frac{
(\lambda e^{-\nu})^2 + (\lambda' e^\nu)^2 
}{2AM}iv
\rp
d\nu
& =
\frac{1}{2} \sqrt{ \frac{AM}{v} }
\sgn(\lambda') \ebf
\lp
\frac{|\lambda \lambda'| iv }{AM}
\rp, \quad \lambda \neq 0.
\end{align}
If $\lambda \in L + h$ has negative norm, then the integral will vanish. 
Thus 
\begin{equation}
\label{eq:vartheta_fe}
\vartheta(\tau, \pm L) = 
 \sum_{h \in L^*/L} \ef_h 
\sum_{\lambda \in \Gamma'_L \backslash L + h, \; \pm Q(\lambda) > 0}
\sgn(\lambda)
\ebf\lp
|Q(\lambda)| \tau
\rp
\end{equation}
is in $S_{1, \rho_{\pm L}}(\SL_2(\Zb))$. 
Hecke noticed that if there exists $\varep < 0$ in $\Gamma_L$, then $\vartheta(\tau, L)$ vanishes identically \cite[Satz 1]{Hecke26}.
Thus, we can suppose $\Gamma'_L = \Gamma_L$.

Another way to show that $\vartheta(\tau, \pm L)$ is holomorphic without explicitly computing the integral in equation \eqref{eq:vartheta} is to apply the lowering operator $\xi_1$ to $\Theta(\tau, \pm L; t) \frac{dt}{t}$ and show that it is an exact form on $\Rb^\times$.
Then its integral over $\Gamma_L \backslash \Rb^\times_+$ would vanish since this locally symmetry domain has no boundary. 
Let $d_t := t \tfrac{\partial}{\partial t}$ be the invariant vector field on $\Rb^\times_+$, i.e.\ an element in the Lie algebra of $\mathrm{O}(V_\Rb)$. 
We have the following proposition, which is a very special case of a result by Kudla and Millson (see section 8 of \cite{KM90}).
\begin{prop}
  \label{prop:diff_ops}
For all $\tau \in \Hc$ and $t \in \Rb^\times$, we have
\begin{equation}
  \label{eq:diff_ops}
  \xi_1 {L_\tau \Theta(\tau, \pm L; t)} = - d_t \Theta(\tau, \mp L; t).
\end{equation}
\end{prop}

\begin{proof}
We will show the equality with $L$ on the left hand side and $-L$ on the right hand side.
The other combination of signs can be proved similarly.
  Straightforward calculations show that
$$
2 \xi_1 \phi_\tau (x, y) = \frac{x}{y} (1 - 4\pi v y^2) \phi_\tau (y, x), \;
x \frac{\partial \phi_\tau}{\partial x} = (1 + 2 \pi i x^2 \tau) \phi_\tau, \;
 \frac{\partial \phi_\tau}{\partial y} = -2 \pi i y \overline{\tau} \phi_\tau, \;
d_t Z^\pm_t = Z^\mp_t.
$$
This implies that for all $X \in \VR$,
\begin{equation}
\label{eq:diff_Schwartz}
  2 \xi_1 \phi_\tau(B(X, Z^+_t), B(X, Z^-_t)) = 
  d_t \phi_\tau(B(X, Z^-_t), B(X, Z^+_t)).
\end{equation}
Since $\phi_\tau(B(\iota(X), Z^+_t), B(\iota(X), Z^-_t))  =  - \phi_\tau(B(X, Z^-_t), B(X, Z^+_t))$, equation \eqref{eq:diff_ops} follows immediately from the definition of $\Theta(\tau, \pm L; t)$ in equations \eqref{eq:Theta}, \eqref{eq:Theta-} and \eqref{eq:diff_Schwartz}.
\end{proof}

\begin{rmk}
  The calculations above can be made to work when $L$ is \textit{isotropic} (see \cite{Kudla81}). In that case, the resulting modular form is an Eisenstein series of weight one.
\end{rmk}

  \section{A Special Function.}
\label{sec:special_func}

In this section, we will introduce a continuous function $g_\tau \in L^p(\Rb)$ for $1 \le p \le \infty$ and use it later as a replacement for the usual Schwartz function to form a real-analytic theta series. 
The inspiration of this function comes from Zwegers' work on the Appell-Lerch sum \cite{ZwThesis}.
After adding a non-holomorphic function, the Appell-Lerch sum becomes a modular object (a real-analytic Jacobi form). Jacobi theta functions are prototypical examples of such modular objects and are constructed by averaging a Schwartz function over a lattice. In view of this, it is very natural to expect that the completed Appell-Lerch sum can be decomposed into the average of a special function over a lattice. This line of thought eventually brought us to the function $g_\tau$, which satisfies properties similar to the function $W_z$ in Proposition 1 of Section 2.3 of \cite{HZ76} (there $z$ is in the upper half plane, as our $\tau$ here).
Using $W_z$, Hirzebruch and Zagier constructed a non-holomorphic modular form of weight 2, which is essentially the product of the weight $3/2$ non-holomorphic Eisenstein series and weight $1/2$ theta series. 
Its holomorphic projection is then a modular form, which is the generating function of intersections of Hirzebruch-Zagier divisors. 
In our setting, the function $g_\tau$ will be used to produce a harmonic Maass form of weight $1/2$.
In Section \ref{sec:tTheta}, we will use it to construct a preimage $\tTheta(\tau, L)$ of $\Theta(\tau, L; 1)$ under $\xi_1$, which turns out to be a sum of products of weight $\half$ harmonic Maass forms and weight $\half$ theta functions (see \eqref{eq:tTheta1}).

\subsection{Fourier transform and distribution.}
\label{subsec:FT}
For $1 \le p \le \infty$, let $L^p(\Rb)$ be the space of bounded functions on $\Rb$ with respect to the $L^p$-norm, and $\Sc(\Rb) \subset L^p(\Rb)$ the subspace of Schwartz functions.
Denote the Fourier transform of $g \in L^1(\Rb)$ by 
\begin{equation}
  \label{eq:FT}
  \Fc(g)(x) :=
   \int_{\Rb} g(w) \ebf(wx) dw,
 \end{equation}
 with inverse $\Fc^{-1}: \Sc(\Rb) \to \Sc(\Rb)$ given by $\Fc^{-1}(\varphi)(x) := \Fc(\varphi)(-x)$ for $\varphi \in \Sc(\Rb)$.
In addition, we also have the following standard linear operators $\partial, \mu: \Sc(\Rb) \to \Sc(\Rb)$
\begin{equation}
  \label{eq:partialmu}
  \partial( \varphi)(x) := \frac{d \varphi(x)}{dx}, \; 
\mu( \varphi)(x) := 2\pi i x  \varphi(x),
\end{equation}
which commutes with $\Fc$ as follows
\begin{equation}
  \label{eq:FT_basic_property}
   \Fc \circ \mu = \partial \circ \Fc, \;
   \Fc \circ \partial = - \mu \circ \Fc.
\end{equation}
The Fourier transform of the Gaussian $e^{-\alpha \pi x^2}$ is $e^{-\pi x^2/\alpha}/\sqrt{\alpha}$ for any $\alpha \in \Cb$ with $\Re(\alpha) > 0$. 

For $\tau \in \Hc$, we define another linear operator $\Ec_\tau: \Sc(\Rb) \to \Sc(\Rb)$ 
\begin{equation}
  \label{eq:Ectau}
  \Ec_\tau(\varphi)(x) :=  \ebf \lp \frac{x^2 }{2} \tau \rp \int^x_0 \ebf \lp - \frac{w^2}{2} \tau \rp  \varphi(w) dw.
\end{equation}
It is a bijection with the inverse given by
\begin{equation}
\label{eq:Ectau-1}
\Ec_\tau^{-1}(\varphi)(x) := \ebf \lp \frac{x^2 }{2} \tau \rp 
\frac{d}{dx} \lp  \ebf \lp - \frac{x^2}{2} \tau \rp  \varphi(x) \rp =
(- \tau \cdot \mu + \partial)(\varphi)(x).
\end{equation}
It turns out that the conjugate of $\Ec_{-1/\tau}$ by $\Fc$ equals to $\tau \Ec_{\tau}$.

\begin{lemma}
  \label{lemma:conjugate_op}
For any $\tau \in \Hc$ and $\varphi \in \Sc(\Rb)$, we have
\begin{equation}
  \label{eq:conjugate_op}
  \Fc \circ \Ec_{-1/\tau} \circ \Fc^{-1} \circ \Ec_{\tau}^{-1} (\varphi) = \tau \cdot \varphi.
\end{equation}
\end{lemma}
\begin{proof}
  By definition and basic properties of Fourier transform in Eq.\ \eqref{eq:FT_basic_property}, we have
  \begin{align*}
    \Fc^{-1} \circ \Ec^{-1}_\tau  &= 
\Fc^{-1} \circ \lp - \tau \mu  + \partial \rp
= ( \tau \partial + \mu) \circ \Fc^{-1} 
= \tau \Ec^{-1}_{-1/\tau} \circ \Fc^{-1}.
  \end{align*}
from which the claim follows immediately.
\end{proof}

The space of continuous functionals on $\Sc(\Rb)$ is the space of tempered distributions and denoted by $\Sc'(\Rb)$. 
For a distribution $T \in \Sc'(\Rb)$, we can define its derivative $\partial {T} \in \Sc'(\Rb)$ and Fourier transform $\Fc(T) \in \Sc'(\Rb)$ to be the distributions satisfying
$$
\partial {T} (\varphi) := T \lp - \partial {\varphi}{} \rp, \;
  \Fc(T)({\varphi}) := T(\Fc( \varphi))
  $$
for all $\varphi \in \Sc(\Rb)$.
We can embed $L^\infty(\Rb)$ into $\Sc'(\Rb)$ via the map
\begin{equation}
  \label{eq:distribution}
  \begin{split}
\Ic:      L^\infty(\Rb) &\to \Sc'(\Rb) \\
      g &\mapsto \Ic(g) : \varphi \mapsto \int_\Rb \varphi(x) g(x) dx,
      \end{split}
    \end{equation}
whose kernel consists of functions that are zero almost everywhere.
We can expand the input of $\Ic$ to include any measurable function $g$ on $\Rb$ such that the integral $\int_{\Rb} \varphi(x) g(x) dx$ converges absolutely for any $\varphi \in \Sc(\Rb)$.
It is easy to check that $\Ic(\partial g) = \partial \Ic(g)$ if $g$ is differentiable almost everywhere on $\Rb$ with jump singularities, and $\Fc(\Ic(g)) = \Ic(\Fc(g))$ if $g \in L^1(\Rb) \cap L^\infty(\Rb)$.

Let $\delta$ be the Dirac delta distribution, which is characterized by the property that
\begin{equation}
  \label{eq:dirac}
  \delta(\varphi)  =  \varphi(0)
\end{equation}
for all $\varphi \in \Sc(\Rb)$.  
Equivalently, we have $\delta = \partial{ \Ic( H)}{}$, where $H(x) := \tfrac{\sgn(x) + 1}{2} \in L^\infty(\Rb)$ is the Heaviside step function.
One frequently used distribution is the Dirac comb distribution
\begin{equation}
  \label{eq:Sha}
  \Sha(\varphi) := \sum_{n \in \Zb} \varphi(n),
\end{equation}
which satisfies $\Fc(\Sha) = \Sha$ by the Poisson summation formula.
For our purpose, the interesting distribution is in fact the shifted version of $\Sha$ defined by
\begin{equation}
  \label{eq:fh}
  \Sha^+_{}(\varphi) := \sum_{m \in \Zb + \half} \varphi \lp m \rp \ebf \lp \frac{m}{2} \rp.
\end{equation}
It is easy to check that $\Fc(\Sha^+) = i \cdot \Sha^+$.

For any $\tau \in \Hc$, we can define the following even function in $L^\infty(\Rb)$ \begin{equation}
\label{eq:omega}
  \omega_\tau(x) :=
  \ebf \lp - \frac{x^2}{2} \tau \rp \sum_{m \ge |x|, \; m \in \Zb + \half}  
  \ebf \lp \frac{m^2}{2} \tau + \frac{m}{2} \rp,
\end{equation}
which is holomorphic as a function in $\tau \in \Hc$ and satisfies the following key property 
\begin{equation}
  \label{eq:phidiff}
\int_\Rb \varphi(x)
\ebf \lp - \frac{x^2}{2} \tau \rp \frac{d}{dx} \lp \ebf \lp \frac{x^2}{2} \tau \rp \omega_\tau(x) \rp dx  
=
\Sha^+(\varphi)
\end{equation}
for any $\varphi \in \Sc(\Rb)$.
Simple manipulations give us the following lemma.
\begin{lemma}
  \label{lemma:FTinfty}
For any $\tau \in \Hc$ and $\varphi \in \Sc(\Rb)$, we have
  \begin{equation}
    \label{eq:FTinfty}
    \Ic(\omega_\tau)(\varphi) = - \Sha^+(\Ec_\tau(\varphi)), \quad
\Fc(\Ic(\omega_\tau)) = (-i \tau)^{-1} \Ic(\omega_{-1/\tau}).
  \end{equation}
\end{lemma}

\begin{proof}
By definition, we can rewrite
\begin{align*}
    \Ic(\omega_\tau)(\varphi) &= \int_\Rb \lp \partial_x \int^x_0 \ebf \lp - \frac{w^2}{2} \tau \rp \varphi(w) dw \rp
\ebf \lp \frac{x^2}{2} \tau \rp \omega_\tau(x) dx\\
&= \lim_{x_0 \to \infty} \Ec_\tau(\varphi) \omega_\tau \mid^{x_0}_{-x_0} 
-
\int_\Rb \Ec_\tau(\varphi)(x)
\ebf \lp - \frac{x^2}{2} \tau \rp \frac{d}{dx} \lp \ebf \lp \frac{x^2}{2} \tau \rp \omega_\tau(x) \rp dx \\
&= - \Sha^+(\Ec_\tau(\varphi)),
\end{align*}
where the last step follows from Eq.\ \eqref{eq:phidiff}.
The second claim in \eqref{eq:FTinfty} is then a consequence of the formal calculations
\begin{align*}
  \Fc(\Ic(\omega_\tau))(\varphi) &= \Ic(\omega_\tau)(\Fc(\varphi)) = - \Sha^+(\Ec_\tau \circ \Fc(\varphi)) = - \Sha^+(\Fc \circ \Ec_{-1/\tau}(\varphi))/\tau \\
&= - (\Fc(\Sha^+))(\Ec_{-1/\tau}(\varphi))/\tau =  - (-i \tau)^{-1} \Sha^+(\Ec_{-1/\tau}(\varphi)) = (-i \tau)^{-1} \Ic(\omega_{-1/\tau})(\varphi)
\end{align*}
using Lemma \ref{lemma:conjugate_op}.
\end{proof}

\subsection{A special function in $L^p(\Rb)$.}
Before constructing $g_\tau$, we first recall the complementary error function
\begin{equation}
  \label{eq:erfc}
  \erfc(x) := \frac{2}{\sqrt{\pi}} \int^\infty_x e^{-t^2} dt = 1 - \erf(x),
\end{equation}
which decays square exponentially, and a special modular form 
\begin{equation}
  \label{eq:eta3}
  \eta^3(\tau) :=
 -{i} \cdot \sum_{m \in \Zb + \half} m \cdot \ebf \lp \frac{m^2}{2}\tau + \frac{m}{2} \rp
= q^{1/8} \prod_{n \ge 1} ( 1- q^n)^3.
\end{equation}
It satisfies the transformation properties
\begin{equation}
  \label{eq:eta3_transform}
  \eta^3(\tau + 1) = \ebf \lp \frac{1}{8} \rp \eta^3(\tau) ,\;
(-i \tau)^{-3/2} \eta^3(-1/\tau) = \eta^3(\tau)
\end{equation}
and is a modular form of weight $3/2$.

Now for $\tau = u + iv \in \Hc$, we define the following functions
\begin{equation}
  \label{eq:g0}
  \begin{split}
    \gh_\tau(x) &:= 
 \ebf \lp - \frac{x^2}{2} \tau \rp \frac{\sgn(x)}{i \cdot  \eta^3(\tau)} 
\sum_{m > |x|, \;m \in \Zb + \half}
\lp m - |x| \rp \ebf \lp \frac{m^2}{2} \tau + \frac{m}{2} \rp,
\\
\gn_\tau(x) &:= \ebf \lp - \frac{x^2}{2} \tau \rp \frac{\sgn(x)}{2}
 \erfc( \sqrt{2\pi v} |x|),\\
g_\tau(x) &:= \gh_\tau(x) - \gn_\tau(x).
  \end{split}
\end{equation}
The function $\gn_\tau$ decays as a Schwartz function, but is discontinuous at 0. 
On the other hand, $\gh_\tau$ has the same type of discontinuity at 0, but is only piecewise differentiable and does not decay quickly enough. 
Under the operator $\tau \mu + \partial$, the function $\gh_\tau$ is closely related to $\omega_\tau$. 
By themselves alone, these two functions would behave strangely under Fourier transform. 
Fortunately, their difference $g_\tau$ enjoys extremely nice properties.
\begin{prop}
  \label{prop:gtau}
The function $g_\tau$ is in $L^p(\Rb)$ for $1 \le p \le \infty$.
It is odd, continuous on $\Rb$ and differentiable on $\Rb \backslash ( \Zb + \half)$. 
Furthermore for all $x \in \Rb$, it satisfies
\begin{enumerate}
\item $\xi_1{ g_\tau(x)} = \frac{\sqrt{2v}}{2} x \ebf \lp \frac{x^2}{2} \tau \rp$,
\item $g_{\tau + 1}(x) = \ebf \lp - \frac{x^2}{2} \rp g_\tau(x)$,
\item $\Fc(g_{-1/\tau})(x) = -i \sqrt{-i \tau} g_\tau(x)$.
\end{enumerate}
\end{prop}

\begin{proof}
From the definition, it is easy to see that $g_\tau \in L^\infty(\Rb)$ is odd and continuous on $\Rb$.
Also, the functions $2 \ebf(\tfrac{x^2}{2} \tau) \gn_\tau(x) - \sgn(x) = - \erf(\sqrt{2\pi v}x)$ and
$$
2\ebf\lp \frac{x^2}{2} \tau \rp \gh_\tau(x) - \sgn(x) 
=
 \frac{2 i x}{\eta^3(\tau)} \ebf\lp \frac{x^2}{2} \tau \rp \omega_\tau(x)
+
\frac{2 i \sgn(x)}{ \eta^3(\tau)} 
\sum_{\begin{subarray}{c} 0 < m < |x| \\ m \in \Zb + \half \end{subarray}} 
 m  \ebf \lp \frac{m^2}{2} \tau + \frac{m}{2} \rp
$$
are differentiable on $\Rb \backslash (\Zb + \half)$, hence so is their difference.
Here, the function $\omega_\tau$ was defined in Eq.\ \eqref{eq:omega}.
To show that $g_\tau \in L^1(\Rb)$, it suffices to prove that the integral
$$
\int_\Rb (m_w - |w|) e^{-\alpha (m_w^2 - w^2)} dw 
= 2\int^\infty_0  (m_w - w) e^{-\alpha (m_w^2 - w^2)} dw 
$$
converges for any $\alpha > 0$, where 
\begin{equation}
  \label{eq:mw}
m_w :=  \lceil |w| + \tfrac{1}{2} \rceil - \tfrac{1}{2} \in \Zb + \tfrac{1}{2}
\end{equation}
is the least half integer greater than or equal to $|w|$.   
By writing $t := m_w - |w|$ with $t \in [0, 1)$, we have
$$
\int^\infty_0  (m_w - w) e^{-\alpha (m_w^2 - w^2)} dw 
< \sum_{m \in \Zb + \half, m > 0} \int^1_0 t e^{-\alpha t (2m - t)} dt
< \frac{ \zeta(2) e^{\alpha} }{\alpha^2}.
$$
Therefore, $g_\tau \in L^1(\Rb) \cap L^\infty(\Rb) \subset L^p(\Rb)$ for all $1 \le p \le \infty$ and $\Fc(g_\tau) \in L^\infty(\Rb) \cap C^1(\Rb)$.

For the last three claims, (1) and (2) are consequences of straightforward calculations from the definition.
The last claim takes a little more calculations to see and we use the notations in section \ref{subsec:FT}.
Since $g_\tau$ is differentiable on $\Rb \backslash (\Zb + \half)$, we have
\begin{equation}
  \label{eq:diff_identity}
  \ebf \lp - \frac{x^2}{2} \tau \rp  \frac{d}{dx} \lp \ebf \lp \frac{x^2}{2} \tau \rp g_\tau(x) \rp 
=
- \frac{\omega_\tau(x)}{i \eta^3(\tau)} + \sqrt{2v} \ebf \lp - \frac{x^2}{2} \overline{\tau} \rp  =: \gamma_\tau(x) \in L^\infty(\Rb)
\end{equation}
for $x \in \Rb \backslash (\Zb + \half)$.
By the same argument in the proof of Lemma \ref{lemma:FTinfty}, we know that
\begin{equation}
\label{eq:FT1}
\Ic(g_\tau)(\varphi) = - \Ic(\gamma_\tau)(\Ec_\tau(\varphi)).
\end{equation}
Now, Lemma \ref{lemma:FTinfty} and \eqref{eq:eta3_transform} implies 
\begin{equation}
\label{eq:FT1a}
\Fc(\Ic(\gamma_\tau)) =   \sqrt{-i \tau} \Ic(\gamma_{-1/\tau}).
\end{equation}
Proceeding as in the formal argument in the proof of Lemma \ref{lemma:FTinfty} gives us
\begin{align*}
  \Ic(\Fc(g_\tau))(\varphi) &= -\Ic(\gamma_\tau)(\Ec_\tau \circ \Fc(\varphi)) 
= - \Ic(\gamma_\tau) (\Fc \circ \Ec_{-1/\tau} (\varphi))/\tau 
= - \Fc(\Ic(\gamma_\tau))(\Ec_{-1/\tau} (\varphi))/\tau \\
&= -i (-i \tau)^{-1/2} \Ic(\gamma_{-1/\tau})(\Ec_{-1/\tau}(\varphi)) = -i (-i\tau)^{-1/2} \Ic(g_{-1/\tau})(\varphi)
\end{align*}
for any $\varphi \in \Sc(\Rb)$.
This implies that $\Fc(g_\tau) + i (-i\tau)^{-1/2} g_{-1/\tau} \in L^\infty(\Rb)$ vanishes almost everywhere. Since this difference is also continuous, it must be identically zero. Changing $\tau$ to $-1/\tau$ then proves the last claim.
\end{proof}

\subsection{Harmonic Maass form of weight $\frac{1}{2}$.}
\label{subsec:hMf1/2}
In this section, we will use the function $g_\tau$ from the previous section to construct some vector-valued harmonic Maass forms of weight $\half$. 
The basic idea is to treat $g_\tau$ as a Schwartz function and average it over the elements in (a translate of) a lattice.
To make this idea work, we need the following results.
\begin{lemma}
  \label{lemma:denom}
If $x \in \frac{1}{b} \Zb$ for some $b \in 2 \Nb$, then 
$b \cdot \ebf \lp \frac{x^2}{2} \tau \rp  \gh_{\tau}(x)  \in \Zb \llbracket q \rrbracket$ with $q = \ebf(\tau)$ and $\gh_\tau(x)$ is holomorphic in the interior of the upper half plane with $\ord_\infty(\gh_\tau(x)) \ge \frac{|x|}{b} + \frac{1}{2b^2} - \frac{1}{8}$. 
\end{lemma}

 \begin{proof}
The first claim is clear from the expression of $\gh_\tau(x)$ in \eqref{eq:g0}.
For the second claim, notice that the inequality $m > |x|$ implies $m \ge |x| + \frac{1}{b}$ when $m \in \Zb + \half$ and $x \in \frac{1}{b} \Zb$ with $2 \mid b$.
Thus, $m^2 - x^2 \ge \frac{2|x| + 1/b}{b}$ and we obtain the bound on $\ord_\infty (\gh_\tau(x))$. 
The $-\frac{1}{8}$ comes from the $\eta^3(\tau)$ in the denominator, and is responsible for a possible pole at $i \infty$. 
 \end{proof}

\begin{prop}
  \label{prop:converge}
For any fixed $N \in 2 \Nb$ and $\tau \in \Hc$, the infinite sum 
\begin{equation}
\label{eq:inf_sum}
G_\tau(x) := \sum_{n \in \Zb } g_\tau( Nn + x)
\end{equation}
converges absolutely for all $x \in \Rb$ and uniformly for $x$ in compact subsets of $\Rb \backslash (\Zb + \half)$.
For $x_0 \in \Rb$, it satisfies
\begin{equation}
  \label{eq:PSF}
\lim_{x \to x_0} \tau^{-1/2}  G_{-1/\tau}(x)
= \frac{\ebf(1/8)}{N} \sum_{n \in \Zb} g_{\tau} \lp  \frac{n}{N} \rp \ebf \lp \frac{n x_0}{N} \rp.
\end{equation}
For $x_0 \in \Zb + \half$, the function $G_\tau(x)$ satisfies
\begin{equation}
  \label{eq:limit}
  \lim_{x \to x_0} G_\tau(x) = G_\tau(x_0) + \frac{\ebf(x_0/2)}{2\pi N \eta^3(\tau)}.
\end{equation}
\end{prop}

\begin{rmk}
\label{rmk:fail}
If $N \in \Rb \backslash \Nb$, then the argument of the proof will not be sufficient to show the absolute convergence of the sum defining $G_\tau$.
\end{rmk}

\begin{proof}
By definition, $g_\tau = g^+_\tau - g^*_\tau$ and $g^*_\tau$ decays like a Schwartz function. 
For $x \in \Rb$, recall $m_x \in \Zb + \half$ as defined in \eqref{eq:mw} and denote
$$
f_\tau(x):=  \frac{\sgn(x)}{i\eta^3(\tau)} 
(m_x - |x|)
\ebf \lp
\frac{m_x^2 - x^2}{2} \tau + \frac{m_x}{2}
\rp.
$$
Since $\sum_{n \in \Zb} (g_\tau - f_\tau)(Nn + x)$ converges absolutely and uniformly for $x \in \Rb$, it suffices to analyze the convergence of $\sum_{n \in \Zb} f_\tau(Nn + x)$.
When $x \in \Zb + \half$, $f_\tau(Nn + x) = 0$ for all $n \in \Zb$. Otherwise, the minimum distance between $x$ and $\Zb + \half$ is $\epsilon > 0$, and $|f_\tau(Nn + x)| \ll_{\tau, x} e^{-\epsilon n v}$.
Therefore, the sum defining $G_\tau(x)$ converges absolutely and uniformly for $x$ in a compact subset of $\Rb \backslash (\Zb + \half)$.
Note that this argument does not work for $N \in \Rb \backslash \Nb$.

Applying the Poisson summation formula and Prop.\ \ref{prop:gtau} now gives us
$$
G_{\tau}(x)
= \frac{i}{N\sqrt{-i\tau}} \sum_{n \in \Zb} g_{-1/\tau} \lp  \frac{n}{N} \rp \ebf \lp \frac{n x}{N} \rp
$$
for $x \in \Rb \backslash (\Zb + \half)$. The sum on the right hand side converges absolutely and uniformly for $x \in \Rb$ and defines a continuous function in $x$. 
Therefore the equation still holds when taking the limit $x \to x_0$ for $x_0 \in \Zb + \half$. 
Changing $\tau$ to $-1/\tau$ then gives us \eqref{eq:PSF}.

To understand the removable discontinuity of $G_\tau$, notice that both sides of \eqref{eq:limit} are periodic with period $N \in 2\Nb$.
So we can suppose that $x_0 \in (\Zb + \half ) \cap [-N, 0)$ and obtain
\begin{align*}
  \lim_{x \to x_0^+} G_\tau(x) &= G_\tau(x_0) + \lim_{\epsilon \to 0^+} \sum_{n \le 0} f_\tau(Nn + x_0 + \epsilon)
=
G_\tau(x_0) + \frac{\ebf(x_0/2)}{i \eta^3(\tau)}
\lim_{\epsilon \to 0^+} \frac{\epsilon}{1 - \ebf(\epsilon N \tau)}\\
&= 
G_\tau(x_0) + \frac{\ebf(x_0/2)}{2 \pi N \eta^3(\tau)}.
\end{align*}
The limit from the left can be calculated analogously. 
\end{proof}

Now, we are ready to construct some harmonic Maass forms of weight $\half$. 
For any $N \in \Rb$ such that $N^2 \in 2\Nb$, the lattice 
\begin{equation}
\label{eq:PN}
P_N := N \Zb
\end{equation} 
is even integral (with dual lattice $P^*_N = \frac{1}{N} \Zb$) with respect to the quadratic form $Q(x) := \frac{x^2}{2}$. 
We denote by $\rho_N$ the Weil representation associated to $(P_N, Q)$.
The functions
\begin{equation}
  \label{eq:thetaN}
  \begin{split}
      \theta_N(\tau) &:= \sum_{h \in P^*_N/P_N} \theta_{N, h}(\tau) \ef_h, \;
\theta_{N, h}(\tau) := \sum_{r \in P_N + h} \ebf \lp \frac{r^2}{2} \tau \rp,\\
      \thetab_N(\tau) &:= \sum_{h \in P^*_N/P_N} \thetab_{N, h}(\tau) \ef_h, \;
\thetab_{N, h}(\tau) := \sum_{r \in P_N + h} r \ebf \lp \frac{r^2}{2} \tau \rp 
  \end{split}
\end{equation}
are vector-valued modular forms in $M_{1/2, \rho_N}$ and $S_{3/2, \rho_N}$ respectively.

Suppose from now on that $N \in 2\Nb$. Then $P^*_N \subset \Qb$ contains $\Zb + \half$, which is translation invariant under $P_N$. Therefore, it makes sense to consider the cosets $(\Zb + \half)/P_N$. We define
\begin{equation}
  \label{eq:tthetah}
  \tilde{\thetab}_{N, h}(\tau) := 
  \begin{cases}
    \sum_{r \in P_N + h} g_\tau(r), & \text{ if } h \in (P^*_N \backslash (\Zb + \half)) / P_N, \\
  \frac{ \ebf(h/2) E_2(\tau)}{12 N i \eta^3(\tau)} + \sum_{r \in P_N + h} g_\tau(r), & \text{if } h \in (\Zb + \half) / P_N.
  \end{cases}
\end{equation}
Here, $E_2(\tau) := 1 - 24 \sum_{n \ge 1} \sigma_1(n) q^n$ is the holomorphic Eisenstein series of weight 2 that satisfies the cocycle relation
\begin{equation}
\label{eq:E2cocycle}
E_2(\tau) - \tau^{-2} E_2(-1/\tau) = \frac{6i}{\pi \tau}.
\end{equation}
Since $g_\tau$ is odd, we have $\tilde{\thetab}_{N, -h}(\tau) = -\tilde{\thetab}_{N, h}(\tau)$.
The main result of this section is as follows.
\begin{thm}
  \label{thm:hMf1/2}
Suppose that $N \in 2\Nb$.
Then the function $\tilde{\thetab}_N := \sum_{h \in P^*_N/P_N} \tilde{\thetab}_{N, h} \ef_h$ satisfies $\xi_{1/2} \tilde{\thetab}_N = \frac{1}{\sqrt{2}} \thetab_N$ and is a harmonic Maass form in $H_{1/2, \overline{\rho_N}}$.
Furthermore, the holomorphic part of $\tilde{\thetab}_N$ has rational Fourier coefficient with denominator bounded by $12N$.
\end{thm}

\begin{proof}
Since the sum defining $\tilde{\thetab}_N$ converges absolutely by Prop.\ \ref{prop:converge}, we can apply $\xi_{1/2}$ to $\tilde{\thetab}_N$ termwisely and conclude that $\xi_{1/2} \tilde{\thetab}_N = \frac{\sqrt{2}}{2} \thetab_N$ from Prop.\ \ref{prop:gtau}.
The bound on the denominator then follows from Lemma \ref{lemma:denom}, which also implies that $\tilde{\thetab}_N$ is regular on $\Hc$ and has at most linear exponential growth near the cusps.
To show that $\tilde{\thetab}_{N} \in H_{3/2, \overline{\rho_N}}$, we just need to check
\begin{equation}
\label{eq:transform_claim}
\tilde{\thetab}_{N, h}(\tau + 1) = \ebf ( - h^2/2) \tilde{\thetab}_{N, h}(\tau), \;
\tau^{-1/2} \tilde{\thetab}_{N, h}(-1/\tau) = 
\frac{\ebf(1/8)}{N} \sum_{t \in P^*_N/P_N} \ebf(ht) \tilde{\thetab}_{N, t}(\tau)
\end{equation}
for all $h \in P^*_N/P_N$.
Since $g_{\tau + 1}(r) = \ebf(-r^2/2) g_\tau(r)$ by Prop.\ \ref{prop:gtau} and $\ebf(-h^2/2) = \ebf(-1/8) = \eta^3(\tau) / \eta^3(\tau + 1)$ if $h \in \Zb + \half$, the first equation holds.

For the second equation, we start with identity \eqref{eq:PSF}.
Suppose $h \in (P_N^* \backslash ( \Zb + \half))/N \Zb$. Then substituting $x_0 = h$ into identity \eqref{eq:PSF} gives us
\begin{align*}
  \tau^{-1/2} \tilde{\thetab}_{N, h}(-1/\tau) &= \frac{\ebf(1/8)}{N}
\lp \sum_{t \in P^*_N/P_N} \ebf(ht) \tilde{\thetab}_{N, t}(\tau) 
- \frac{E_2(\tau)}{12 N i \eta^3(\tau)} \sum_{t \in (\Zb + \half)/P_N} \ebf((h + \half)t)
\rp.
\end{align*}
By choosing $t = n + \half$ with $0 \le n \le N - 1$ as the representatives of the cosets in $(\Zb + \half)/P_N$, we see that the second sum above vanishes and the equation above becomes the second claim in \eqref{eq:transform_claim}.
Suppose $h \in (\Zb + \half)/N\Zb$ and we also denote $h$ to be the representative in $[0, N)$. Then \eqref{eq:PSF} implies that 
\begin{equation}
\label{eq:id1}
\tau^{-1/2} \lim_{x_0 \to h} G_{-1/\tau}(x_0)
= 
\frac{\ebf(1/8)}{N}
\lp \sum_{t \in P^*_N/P_N} \ebf(ht) \tilde{\thetab}_{N, t}(\tau) 
- \frac{\ebf(h/2) E_2(\tau)}{12 \eta^3(\tau)} 
\rp.
\end{equation}
Using \eqref{eq:limit}, we can evaluate the left hand side above as
\begin{align*}
   \lim_{x_0 \to h}  G_{-1/\tau}(x_0)
&= \tilde{\thetab}_{N, h}(-1/\tau) - \frac{\ebf(h/2) E_2(-1/\tau)}{12 N i \eta^3(-1/\tau)}
+ 
\frac{ \ebf(h/2)}{i \eta^3(-1/\tau)} \frac{\tau}{2\pi i N } \\
&=
\tilde{\thetab}_{N, h}(-1/\tau) - \frac{\ebf(h/2) (E_2(\tau) - \frac{6i}{\pi \tau}) \tau^2 }{12 N i \eta^3(-1/\tau)}
+ 
\frac{ \ebf(h/2)}{i \eta^3(-1/\tau)} \frac{\tau}{2\pi i N } \\
&= 
\tilde{\thetab}_{N, h}(-1/\tau) - \frac{\ebf(h/2) E_2(\tau) \tau^2 }{12 N i (- i \tau)^{3/2} \eta^3(\tau)}.
\end{align*}
Substituting this into \eqref{eq:id1} and using the fact that $\frac{\tau^{-1/2} \tau^2}{i (-i\tau)^{3/2}} = \ebf(1/8)$ then finishes the proof.
\end{proof}

\begin{exmp}
  In the simplest case $N = 2$, $P^*_2/P_2 = \half \Zb / 2\Zb$ and 
$\thetab_2(\tau) = \half \eta^3(\tau) ( \ef_{1/2} - \ef_{3/2})$.
Then the holomorphic part of $\tilde{\thetab}_{2, h}(\tau)$, which we denote by $\tilde{\thetab}^+_{2, h}(\tau)$, vanishes for $h = 0, 1$ and
\begin{align*}
  \tilde{\thetab}^+_{2, 1/2}(\tau)
 &= \frac{E_2(\tau)}{24 \eta^3(\tau)} + \sum_{r \in 2\Zb + 1/2} g^+_\tau(r)\\
&= \frac{E_2(\tau)}{24 \eta^3(\tau)} +
 \sum_{\begin{subarray}{c} r \in 2\Zb + 1/2 \\ m \in \Zb + 1/2, m > |r| \end{subarray}} \frac{\sgn(r) (m - |r|)}{i \eta^3(\tau)} \ebf \lp \frac{m^2 - r^2}{2} \tau + \frac{m}{2} \rp
\\
 &= \frac{E_2(\tau)/24 - F_2^{(2)}(\tau)}{\eta^3(\tau)} =
-\frac{q^{-1/8}}{24}(-1 + 45 q + 231 q^2 + 770 q^3 + O(q^4)) =
  -\tilde{\thetab}^+_{2, 3/2}(\tau),
\end{align*}
where $F_2^{(2)}(\tau) := \sum_{b > a > 0, b-a \text{ odd}} a (-1)^b q^{ab/2}$ comes out of the substitution $a = m - |r|, b = m + |r|$. 
The function $\tilde{\thetab}^+_{2, 1/2}$ plays an important role in Matthieu Moonshine \cite{DMZ12, EOT11}. 
\end{exmp}

\section{Construction of $\tTheta(\tau, L)$.}
\label{sec:tTheta}
In this section, we will construct a modular preimage of $\Theta(\tau, L; t)$ under $\xi_1$ at $t = 1$, and we denote it by $\tTheta(\tau, L)$.
Recall from \eqref{eq:Theta} that $\Theta(\tau, L; t)$ is constructed from $\phi_\tau$ defined in \eqref{eq:phi}.
If we let
\begin{equation}
  \label{eq:tphi}
  \begin{split}
\tTheta^*(\tau, L ; t) &:=  \sum_{h \in L^*/L} \tTheta^*_h(\tau, L ; t) \ef_h, \;
  \tTheta^*_h(\tau, L; t) := \sum_{X \in L+h} \tphi^*_\tau(B(X, Z^+_t), B(X, Z^-_t)), \\
  \tphi^*_\tau(x, y) &:= 2 \ebf \lp \frac{y^2}{2} \tau \rp \gn_\tau(x)
=  \ebf \lp \frac{y^2 - x^2}{2} \tau \rp \sgn(x) \erfc(\sqrt{2\pi v}|x|),
  \end{split}
\end{equation}
then Prop. \ref{prop:gtau} implies that $\xi_1 \tTheta^*(\tau, L; t) = - \Theta(\tau, L; t)$.
However, $\tTheta^*(\tau, L; t)$ is \textit{not} modular since $\tphi^*_\tau$ does not behave well under the Fourier transform $\Fc_0$. In view of Prop.\ \ref{prop:gtau}, it is more natural to replace $\gn_\tau$ in \eqref{eq:tphi} with $g_\tau$ to form a modular form, but it might not converge for all $L$ and $t$ (see Remark \ref{rmk:fail}). 

The situation becomes much better at the special point $t = 1$, where the theta kernel $\Theta(\tau, L; t)$ becomes a finite sum of products of weight $\frac{3}{2}$ holomorphic unary theta series with anti-holomorphic theta functions. 
Even though $\tTheta(\tau, L)$ is not harmonic, it still breaks naturally into the sum of a holomorphic part and a non-holomorphic part $\tTheta^*(\tau, L) := \tTheta^*(\tau, L; 1)$.
The holomorphic part is a vector-valued Laurent series in $q^{1/d_L}$, which will have rational Fourier coefficients with explicitly bounded denominators.
This holomorphic function is also called a ``mixed-mock modular form'' in the literature \cite{DMZ12}.
Our method should work for any $t \in F \cap \Rb^\times_+$, and it would be very interesting to construct the preimage for all $t$.
We split the construction into a special case and the general case, and deduce the latter from the former.

\subsection{Special case.}
\label{subsec:special}
Suppose that $L = L_{\af, 2AN^2 }$ for some integral ideal $\af \subset \Oc_D$ and positive integer $N$.
Let $P := 2A^2N^2 (\Zb \oplus \sqrt{D}\Zb) \subset L$ be a sublattice.
Via the map $a + b\sqrt{D} \mapsto (\tfrac{a}{AN}, \tfrac{\sqrt{D} b}{AN})$, the lattice $P$ is isometric to $P_{2AN} \oplus (-P_{2AN\sqrt{D}})$ and $P^*/P \cong P^*_{2AN}/P_{2AN} \times P^*_{2AN\sqrt{D}}/P_{2AN\sqrt{D}}$ as finite abelian groups (see \eqref{eq:PN}).
Therefore, 
\begin{align*}
\Theta_{} (\tau, P; 1) &= 
\frac{2\sqrt{v}}{\sqrt{2}} 
\sum_{h \in P^*/P}
\sum_{    a + b\sqrt{D} \in P + h }
\frac{a }{AN}
\ebf\lp
\frac{(a/(AN))^2 \tau}{2}\rp
\ebf\lp
\frac{ - (\sqrt{D} b/(AN))^2 \overline{\tau}}{2}
\rp\\
&= \sqrt{2v} \thetab_{2AN}(\tau) \otimes \theta_{2AN\sqrt{D}}(-\overline{\tau})
\end{align*}
by \eqref{eq:theta_explicit}, where $\thetab$ and $\theta$ are defined in \eqref{eq:thetaN}.
Recall that there is a $\Gamma$-linear map $\psi : \Cb[P^*/P] \to \Cb[L^*/L]$ defined in \eqref{eq:psi} and a harmonic Maass form $\tilde{\thetab}_{2AN}$ constructed in section \ref{subsec:hMf1/2}.
We can now define
\begin{equation}
  \label{eq:tTheta1}
  \tTheta(\tau, L) := 2 \psi(\tilde{\thetab}_{2AN}(\tau) \otimes \theta_{2AN\sqrt{D}}(\tau))
\end{equation}
and prove the following result
\begin{prop}
  \label{prop:tTheta}
The function $\tTheta(\tau, L)$ is a real-analytic modular form in $\Ac_{1, \rho_{-L}}$ such that $\xi_1 (\tTheta(\tau, L) ) = \Theta(\tau, L; 1)$.
Furthermore, its holomorphic part
\begin{equation}
  \label{eq:denom_Theta}
 \tTheta^+(\tau, L) := 
2 \psi(\tilde{\thetab}^+_{2AN}(\tau) \otimes \theta_{2AN\sqrt{D}}(\tau))
\end{equation}
is a formal Laurent series in $\frac{1}{6AN}\Zb[L^*/L](\!( q^{1/d_L} )\!)$
\end{prop}

\begin{proof}
Since $\psi$ is defined over $\Zb$, it commutes with $\xi$ and is $\Gamma$-linear with respect to $\rho_{-P}$ and $\rho_{-L}$ as well. Therefore, $\tTheta(\tau, L) \in \Ac_{1, \rho_{-L}}$ and 
$$
\xi_1 \tTheta(\tau, L) = 2\sqrt{v} \psi \lp \xi_{1/2} ( \tilde{\thetab}_{2AN}(\tau)  ) \otimes \theta_{2AN\sqrt{D}}(-\overline{\tau})
\rp
= \psi( \Theta(\tau, P; 1)) = \Theta(\tau, L; 1)
$$
by \eqref{eq:equivariance} and Theorem \ref{thm:hMf1/2}, which also implies the bound on the denominator.
\end{proof}

\subsection{General case.}
\label{subsec:general}
For any even, integral lattice $(L, Q)$ and $N \in \Nb$, denote the scaled lattice $(NL, \frac{Q}{N})$ by $NL$. Note that they have the same dual lattice and there is a natural projection 
$$
L^*/(NL) \to L^*/L.
$$
The following simple lemma then relates $\Theta(\tau, L; t)$ and $\Theta(\tau, NL; t)$.

\begin{lemma}
  \label{lemma:average}
Fix $\tau \in \Hc$ and $t \in \Rb^\times_+$. For any $h \in L^*/L$, we have
$$
\Theta_h(\tau, L; t) = \sum_{\delta \in L^*/NL, \; \delta \equiv h \bmod L}
\Theta_\delta(N \tau, NL; t). 
$$
Equivalently, let $\Cc_{L, N}$ be the $|L^*/L| \times |L^*/NL|$ matrix defined by 
\begin{equation}
\label{eq:Cc}
\Cc_{L, N} := \lp \mathds{1}_{L}(h - \delta) \rp_{h \in L^*/L, \delta \in L^*/NL},
\end{equation}
where $\mathds{1}_L$ is the characteristic function of $L$.
Then 
\begin{equation}
  \label{eq:scale_id}
\Theta(\tau, L; t) =  \Cc_{L, N} \cdot \Theta(N\tau, NL; t).
\end{equation}
\end{lemma}

\begin{rmk}
  This already appeared in Hecke's work \cite[Eq.\ III in \textsection 4]{Hecke26}.
\end{rmk}

Let $L = L_{\af, M}$ with $\af \subset \Oc_D$ an integral ideal and $M \in \Nb$ any natural number. 
Define $N \in \Nb$ to be the smallest positive integer such that 
\begin{equation}
  \label{eq:N}
NM = 2A(N')^2
\end{equation}
for some $N' \in \Nb$.
In particular, we can always choose $N = 2AM$ and $N' = M$. 
Using the function $\tTheta(\tau, NL)$ in \eqref{eq:tTheta1}, we can construct $\tTheta(\tau, L)$ with the proposition below.
\begin{prop}
  \label{prop:GammaN}
The function $\Cc_{L, N} \cdot \tTheta(N\tau, NL)$ maps to $\Theta(\tau, L; 1)$ under $\xi_1$ and is in $\Ac_{1, \rho_{-L}}(\Gamma_0(N))$.
\end{prop}

\begin{proof}
The first claim follows from equation \eqref{eq:scale_id} and $\xi_1 \tTheta(N \tau, NL) = \Theta(N \tau, NL)$.
For the second claim, notice that for $\gamma \in \Gamma_0(N)$, we have
\begin{equation}
\label{eq:slash_id}
\lp \Cc_{L, N} \cdot  \tTheta(N\tau, NL) \rp \mid_{1, \rho_{-L}} \gamma = 
\rho^{-1}_{-L}(\gamma) \cdot \Cc_{L, N} \cdot \rho_{-NL}(\gamma_N) \cdot \tTheta(N \tau, NL),
\end{equation}
where $\gamma_N:= \smat{N}{}{}{1} \cdot \gamma \cdot \smat{1/N}{}{}{1} \in \Gamma$. 
Thus, it suffices to show the matrix identity
\begin{equation}
    \label{eq:weil_id}
\rho_{-L}(\gamma) \cdot \Cc_{L, N} = \Cc_{L, N} \cdot \rho_{-NL}(\gamma_N)
\end{equation}
for all $\gamma \in \Gamma_0(N)$. 
Denote
$$
M_\gamma := \rho_{L}(\gamma) \cdot \Cc_{L, N} - \Cc_{L, N} \cdot \rho_{NL}(\gamma_N).
$$
It suffices to prove that $\ef_h$ is in the right kernel of $M_\gamma$ for all $h \in L^*/NL$. 
For $t \in \Rb^\times_+$, consider the theta series
$$
\theta(\tau, NL; t) := \sqrt{v} \sum_{h \in L^*/NL} \ef_h \sum_{\lambda \in NL + h} 
\ebf \lp 
\frac{(\lambda t^{-1} + \lambda't)^2}{2AMN} \tau - \frac{(-\lambda t^{-1} + \lambda't)^2}{2AMN} \overline{\tau}
\rp \in \Ac_{0, \rho_{NL}}(\Gamma). 
$$
Then $\Cc_{L, N} \cdot \theta(\tau, NL; t) = \theta(\tau, L; t)$ and for all $\gamma \in \Gamma_0(N), t \in \Rb^\times_+$ 
$$
M_\gamma \cdot \theta(N\tau, NL; t) = 0
$$
as a $\Cb[L^*/L]$-valued function on $\Hc$. 
This follows from the same calculations that produced equation \eqref{eq:slash_id}.
This power series identity necessarily becomes an identity between the Fourier coefficients, which are functions of $v$. 
From the asymptotic behavior with respect to $v$, we can deduce that $\ef_h$ is in the kernel of $M_\gamma$ for all $h \in L^*/NL$, i.e.\ $M_\gamma$ vanishes identically. Applying complex conjugation then gives us equation \eqref{eq:weil_id}.
\end{proof}

Now, we can average $\Cc_{L, N} \cdot \tTheta(N\tau, NL)$ over $\Gamma_0(N) \backslash \Gamma$ to define
\begin{equation}
  \label{eq:tTheta_gen}
  \tTheta(\tau, L) := \frac{1}{[\Gamma: \Gamma_0(N)]} \sum_{\gamma \in \Gamma_0(N) \backslash \Gamma} \lp \Cc_{L, N} \cdot \tTheta(N \tau, NL) \rp \mid_{1, \rho_{-L}} \gamma.
\end{equation}
The main result of this section is as follows.
\begin{thm}
  \label{thm:tTheta_gen}
Let $L = L_{\af, M}$. The function $ \tTheta(\tau, L) \in \Ac_{1, \rho_{-L}}(\Gamma)$ is a real-analytic automorphic form such that $\xi_1(\tTheta(\tau, L)) = \Theta(\tau, L; 1)$ and $\tTheta^+(\tau, L) := \tTheta(\tau, L) + \tTheta^*(\tau, L)$ is in $\frac{1}{\kappa_L} \Zb[L^*/L] (\!( q)\!)$ with 
\begin{equation}
\label{eq:kappaL}
\kappa_L :=  12 A^3 (N')^3  \cdot  \phi(N),
\end{equation}
where $\phi(N) := [\Gamma: \Gamma_0(N)] = N \prod_{p \mid N \text{ prime}}(1 + \frac{1}{p})$.
In particular $\kappa_L$ can be chosen to divide $12 (AM)^3 \phi(2AM)$.
\end{thm}


\begin{proof}
Since $\xi_1$ commutes with $\mid_1$ and conjugates $\rho_{-L}$ to $\rho_{L}$, we obtain the first claim from Proposition \ref{prop:GammaN}.

To prove the second claim, we will first show that $\tTheta^+(\tau, L) \in \Qb[L^*/L]\pars{q}$, then give a bound of the denominator.
For $\gamma = \smat{*}{*}{c}{*} \in \Gamma$, we can write $N_\gamma := \gcd(N, c)$ and 
\begin{equation}
\label{eq:gammaN}
\pmat{N}{}{}{1} \gamma = \gamma_N \cdot \pmat{N_\gamma}{b}{0}{N/N_\gamma}, \; b \in \Zb, \; \gamma_N \in \Gamma.
\end{equation}
Then for every $\gamma \in \Gamma, \tau \in \Hc$ and $f \in \Ac_{1, \rho}(\Gamma)$, we have
\begin{equation}
\label{eq:mat_iden}
  f(N\tau) \mid_{1} \gamma =
\frac{N_\gamma}{N} 
\rho(\gamma_N) \cdot  f(\tau_\gamma), \;
\tau_\gamma := (\gamma_N^{-1} \smat{N}{}{}{1} \gamma) \cdot \tau = \frac{N_\gamma \tau + b}{N/N_\gamma} \in \Hc.
\end{equation}
Now applying this to $f(\tau) = \tTheta(\tau, NL) \in \Ac_{1, \rho_{-NL}}(\Gamma)$ gives us
$$
\tTheta(N\tau, NL) \mid_{1} \gamma =
\frac{N_\gamma}{N} 
\rho_{-NL}(\gamma_N) \cdot  \tTheta^+(\tau_\gamma, NL).
$$
Substituting this into the definition of $\tTheta(\tau, L)$ gives us
\begin{equation}
\label{eq:average1}
\tTheta(\tau, L) = \frac{1}{ [\Gamma: \Gamma_0(N)]}
\sum_{\gamma \in \Gamma_0(N) \backslash \Gamma} \frac{N_\gamma}{N}
\rho_{-L}^{-1}(\gamma) \cdot 
 \Cc_{L, N} \cdot 
\rho_{-NL}(\gamma_N) 
 \tTheta(\tau_\gamma, NL).
\end{equation}

Notice that we have the following (rather cute) linear algebra identity relating the non-holomorphic parts 
\begin{equation}
  \label{eq:la_id}
\tTheta^*(\tau, L) =  \frac{N_\gamma}{N} \rho^{-1}_{-L}(\gamma) \cdot \Cc_{L, N} \cdot 
\rho_{-NL}(\gamma_N) \tTheta^*(\tau_\gamma, NL), \; \gamma \in \Gamma.
\end{equation}
To prove this identity, we first apply equation \eqref{eq:mat_iden} to $f(\tau) = \Theta(\tau, NL; 1) \in \Ac_{1, \rho_{NL}}(\Gamma)$ to obtain
$$
\Theta(\tau, L; 1) =  \frac{N_\gamma}{N} \rho^{-1}_{L}(\gamma) \cdot \Cc_{L, N} \cdot 
\rho_{NL}(\gamma_N) \Theta(\tau_\gamma, NL; 1)
$$
for all $\gamma \in \Gamma$.
Since $\xi_1(\tTheta^*(\tau_\gamma, NL)) = \Theta(\tau_\gamma, NL; 1)$, we see that the difference between the two sides of equation \eqref{eq:la_id} vanishes under $\xi_1$, implying that it is holomorphic. 
Furthermore, this difference is stable under $\tau \mapsto \tau + N$ and vanishes as $v \to \infty$, so has a Fourier expansion $\sum_{n \ge 1} a_n \ebf(n\tau/N)$. However, integrating this difference against $\ebf(-nu/N)$ over $u \in [0, N]$ equals to zero for all $n \ge 1$, which means $a_n = 0$ for all $n \ge 1$. Thus the difference vanishes identically. 
Now substituting $\tTheta(\tau_\gamma, NL) = \tTheta^+(\tau_\gamma, NL) - \tTheta^*(\tau_\gamma, NL)$ and identity \eqref{eq:la_id} into equation \eqref{eq:average1}, we can write
\begin{equation}
  \label{eq:tTheta+}
  \tTheta^+(\tau, L) = \frac{1}{[\Gamma: \Gamma_0(N)]}
\sum_{\gamma \in \Gamma_0(N) \backslash \Gamma} \frac{N_\gamma}{N}
\rho_{-L}^{-1}(\gamma) \cdot 
 \Cc_{L, N} \cdot 
\rho_{-NL}(\gamma_N) \cdot \tTheta^+(\tau_\gamma, NL).
\end{equation}

The Weil representations $\rho_{-L}$ and $\rho_{-NL}$ are defined over $\Qb(\zeta_{N^\dagger})/\Qb$ with $N^\dagger = ADMN$ and $\zeta_{N^\dagger}$ a primitive $(N^\dagger)$\tth root of unity.
For $a \in (\Zb/N^\dagger \Zb)^\times$, recall that $J_a = \smat{1}{}{}{a} \in \GL_2(\Zb/N^\dagger \Zb)$ and $\sigma_a \in \Gal(\Qb(\zeta_{N^\dagger})/\Qb)$ be the corresponding element as in Section \ref{subsec:weilrep}.
Let $\gamma' \in \Gamma$ be any element such that its image in $\Gamma(N^\dagger) \backslash \Gamma  \cong \SL_2(\Zb/N^\dagger \Zb)$ is $J_a^{-1} \gamma J_a$. Then $N_\gamma = N_{\gamma'}$ and we can write
$$
\pmat{N}{}{}{1} \gamma' = \gamma'_N \cdot \pmat{N_\gamma}{ab}{0}{N/N_\gamma}
$$
with the image of $\gamma'_N \in \Gamma$ in $\SL_2(\Zb/N^\dagger \Zb)$ being $J_a^{-1} \gamma_N J_a$.

Since $\tTheta^+(\tau, NL) \in \Qb[L^*/NL]\pars{q^{1/d_{NL}}}$ and $\tau_\gamma = \frac{N_\gamma \tau + b}{N/N_\gamma}$ with $b \in \Zb$, 
we have
$$
\sigma_a \tTheta^+(\tau_\gamma, NL) = 
\tTheta^+ \lp \frac{N_\gamma \tau + ab}{N/N_\gamma}, NL \rp 
=
\tTheta^+ \lp \tau_{\gamma'}, NL \rp, \;
\tau_{\gamma'} := \lp (\gamma'_N )^{-1} \pmat{N}{}{}{1} \gamma' \rp \cdot \tau.
$$
Since $d_{-L} \mid d_{-NL} \mid N^\dagger$, the representations $\rho_{-L}$ and $\rho_{-NL}$ are trivial on $\Gamma(N^\dagger)$.
By Proposition \ref{prop:rationality}, we have
$$
\sigma_a \lp N_\gamma \rho_{-L}^{-1}(\gamma) \cdot \Cc_{L, N} \cdot
\rho_{-NL}(\gamma_N) \cdot \tTheta^+(\tau_\gamma, NL)
\rp
=
N_{\gamma'} \rho_{-L}^{-1}(\gamma') \cdot \Cc_{L, N} \cdot
 \rho_{-NL} (\gamma'_N)
\cdot 
\tTheta^+ \lp \tau_{\gamma'}, NL \rp.
$$
Thus, $\sigma_a$ permutes the summands on the right hand side of equation \eqref{eq:tTheta+}, which means that $\tTheta^+(\tau, L)$ has Fourier coefficients in $\Qb$. 


From the explicit formula of Weil representation in \cite[Theorem 4.7]{Sch09}, we know that the denominator of every entry in $\rho_{-L}$, resp.\ $\rho_{-NL}$, is bounded by $\sqrt{AM}$, resp.\ $N\sqrt{AM}$. 
Proposition \ref{prop:tTheta} and the choice of $N$ in equation \eqref{eq:N} tells us that the denominator of $\tTheta^+(\tau', NL)$ is bounded by $6AN'$. 
Thus, the denominator of $\tTheta^+(\tau, L)$ is bounded by $\kappa_L$.
\end{proof}

\section{Deformation of Theta Integral.}
\label{sec:tvartheta}

In this section, we will construct a harmonic Maass form with the following property.
\begin{thm}
  \label{thm:Main}
In the notation of Section \ref{sec:lift}, let $M \in \Nb$, $D \ge 1$ a discriminant, $\af \subset \Oc_D$ an integral ideal with $A:= [\Oc_D : \af]$, and $\vartheta(L, \tau)$ the vector-valued cusp form associated to $(L, Q) = (M\af, \frac{\Nm}{AM})$.
There exists a harmonic Maass form $\tvartheta(\tau, L_{}) = \sum_{h \in L^*/L} \tvartheta_h(\tau, L) \ef_h$ in $H_{1, \rho_{-L}}(\SL_2(\Zb))$ such that $\xi_1(\tvartheta(\tau, L)) = \vartheta(\tau, L)$ and its holomorphic part $\tvartheta^+_h(\tau, L)$ has the Fourier expansion
$$
\tvartheta^+_h(\tau, L) = \sum_{n \in \Qb, n \gg -\infty} c^+_L(n, h) q^n,
$$
with the Fourier coefficient $c_L^+(n, h)$ satisfying
\begin{equation}
\label{eq:coeff}
\begin{split}
c^+_L(n, h) &-  \sum_{\begin{subarray}{c} \lambda \in \Gamma_L \backslash L + h \\ -Q(\lambda) = n > 0 \end{subarray}} \sgn(\lambda) \log \left| \frac{\lambda}{\lambda'} \right| \in \frac{1}{\kappa_{}} \Zb \cdot  \log \varep_L
\end{split}
\end{equation}
for an explicit constant  $\kappa_{} \in \Nb$ depending on $D$ and $M$ only.
In particular, when $\af$ is a proper $\Oc$-ideal and $\gcd(A, M) = 1$, then one can choose $\kappa_{}$ to divide $24 M^3 \phi(2M)$, where $\phi$ is the multiplicative function defined in Theorem \ref{thm:Main_sc_1}.
\end{thm}

\begin{rmk}
  \label{rmk:choice}
The summation in equation \eqref{eq:coeff} is finite and the choice of representative $\lambda \in \Gamma_L \backslash L + h$ does not affect the statement of the result.
\end{rmk}

\begin{rmk}
Using a trick with Stokes' theorem, the theorem above immediately implies that the Petersson norm of $\vartheta(\tau, L)$ is in $\frac{1}{\kappa} \Zb \cdot \log \varep_L$.
Using the Fourier-Jacobi expansion of theta integrals by Kudla \cite{Kudla16}, the second author has shown that this norm is in fact always in $\frac{1}{12} \Zb \cdot \log \varep_L$ \cite{Li17b}.
\end{rmk}

\begin{rmk}
  When $L$ is isotropic, a similar statement holds, and the Fourier coefficients are logarithms of rational numbers (see \cite{Li17a}).
\end{rmk}

The starting point of the construction is the deformed integral $I(\tau, -L, s)$ defined by
\begin{equation}
\label{eq:s_int}
 I(\tau, -L, s):= \int_{1}^{\varep_L}   t^s \Theta(\tau, -L; t) \frac{dt}{t}.
\end{equation}
Since $[1, \varep_L]$ is compact, $I(\tau, -L, s)$ is holomorphic for $s \in \Cb$ and has the following Taylor series expansion at $s = 0$
\begin{equation}
\label{eq:Laurent}
I(\tau, -L, s) = \vartheta(\tau, -L) + I'(\tau, -L) s + O(s^2),
\end{equation}
where $I'(\tau, -L) = \frac{\partial}{\partial s} I(\tau, -L, s) \mid_{s = 0}$. By Proposition \ref{prop:diff_ops}, applying $\xi_1$ to $I(\tau, L,  s)$ gives us
\begin{align*}
\xi_1 I(\tau, -L, s) &= -\frac{1}{2} \int^{\varep_L}_1   t^s d_t \Theta(\tau, L; t) \frac{dt}{t} 
= \frac{1 - \varep_L^s}{2} \Theta(\tau, L; 1)  + \frac{s}{2}  I(\tau, L, s)
\end{align*}
after using integration by parts, which implies
\begin{equation}
\label{eq:key_eq}
2 \xi_1(I'(\tau, -L)) + \log \varep_L \cdot \Theta(\tau, L; 1) = \vartheta(\tau, L).
\end{equation}
Recall that $\tTheta(\tau, L) \in \Ac_{1, \rho_{-L}}(\Gamma)$ is the preimage of $\Theta(\tau, L; 1)$ under $\xi_1$ as in Theorem \ref{thm:tTheta_gen}. Then $2I'(\tau, -L) + \log \varep_L \cdot \tTheta(\tau, L) \in \Ac_{1, \rho_{-L}}(\Gamma)$ is a preimage of $\vartheta(\tau, L)$ under $\xi_1$. 
In the rest of this section, we will calculate its Fourier expansion and prove Theorem \ref{thm:Main}.

\subsection{Fourier expansion of $I'(\tau, -L)$.}
Now, we will calculate the Fourier expansion of 
\begin{equation}
\label{eq:I'h}
I'_h(\tau, -L) := \int^{\varep_L}_1 \log t \cdot \Theta_h(\tau, -L; t) \frac{dt}{t}
\end{equation}
for each $h \in L^*/L$. 
To state the main result, we will first setup a few notations and make some choices.
For simplicity, we write $\varep = \varep_L$. For $\lambda \neq 0$, let
\begin{equation}
\label{eq:ratio}
r(\lambda) := \left| \frac{\lambda}{\lambda'} \right|.
\end{equation}
For each orbit $\Lambda \in \Gamma_L \backslash L + h$ with $Q(\Lambda) \neq 0$, we fix a representative $\lambda_0 \in \Lambda$ such that 
\begin{equation}
  \label{eq:rep}
  1 \le  r(\lambda_0) < \varep^2.
\end{equation}
%
After such a representative has been fixed, we will set $ \sgn(\Lambda) := \sgn(\lambda_0)$,
\begin{equation}
\label{eq:aLambda}
a(\Lambda) := 
\begin{cases}
\sgn(\lambda_0) \log r(\lambda_0), & r(\lambda_0) \neq 1, \\
- {\log \varepsilon}{}, & r(\lambda_0) = 1,
\end{cases}
\end{equation}
and use the convenient notation 
\begin{equation}
  \label{eq:lambdan}
\lambda_n := \lambda_0 \varep^n.  
\end{equation}
Notice that $-2 \log \varepsilon < a(\Lambda) < 2\log \varepsilon$.
We can now state the main result of this section.
\begin{prop}
  \label{prop:FE_I'}
The function $I'_h(\tau, -L)$ has the Fourier expansion
\begin{equation}
  \label{eq:FE_I'}
  2I'_h(\tau, -L) = \sum_{\begin{subarray}{c} \Lambda \in \Gamma_L \backslash L + h \\ Q(\Lambda) < 0 \end{subarray}} a(\Lambda) \ebf(-Q(\Lambda)\tau)
-
\tvartheta^*_h(\tau, L)
- \log \varepsilon \cdot  \tTheta^*_h(\tau, L),
\end{equation}
where $ \tvartheta^*_h(\tau, L)$ and $ \tTheta^*_h(\tau, L)$ are given by
\begin{align*}
  \tvartheta^*_h(\tau, L) &= \sum_{\begin{subarray}{c} \lambda \in \Gamma_L \backslash L + h\\ Q(\lambda) > 0 \end{subarray}} \sgn(\lambda) \Gamma(0, 4 \pi Q(\lambda) v) q^{-Q(\lambda)}, \;
\tTheta^*_h(\tau, L) = \sum_{\lambda \in L + h} \tphi^*_\tau \lp \frac{\lambda + \lambda'}{\sqrt{2AM}}, \frac{\lambda - \lambda'}{\sqrt{2AM}} \rp.
\end{align*}
\end{prop}

\begin{proof}
Substituting in equation \eqref{eq:theta_explicit} gives us
\begin{align*}
  I'_h(\tau, -L) &= 
\sqrt{\frac{v}{AM} } 
\sum_{\begin{subarray}{c} \Lambda \in \Gamma_L \backslash L + h \\ Q(\Lambda) \neq 0 \end{subarray}} 
\ebf \lp -Q(\Lambda) u \rp
\sum_{\lambda \in \Lambda}
\int^\varep_1
\lp
\lambda t^{-1} - \lambda't 
\rp
\ebf\lp
\frac{\lp (\lambda t^{-1})^2 + (\lambda' t)^2  \rp iv }{2AM}
\rp
\log t \frac{dt}{t} 
\end{align*}
Since $\Lambda = \{\lambda_0 \varep^n: n \in \Zb\}$, each sum over $ \Lambda$ becomes
\begin{align*}
\sum_{n \in \Zb}
\int^\varep_1 
\lp
\lambda_n t^{-1} - \lambda'_n  t 
\rp
\ebf\lp
\frac{\lp (\lambda_{n} t^{-1})^2 + (\lambda'_n t)^2  \rp iv }{2AM}
\rp
\log t \frac{dt}{t} .
\end{align*}
In each summand, let $\nu :=  -\frac{\log r(\lambda_n)}{2} + \log t $. Then, we can write $\log t = \nu  + \frac{\log r(\lambda_0)}{2} + n \log \varep $ and break $I'_h(\tau, -L)$ into three pieces
\begin{equation}
  \label{eq:I'h_2}
2  I'_h(\tau, -L) =
\sum_{\begin{subarray}{c} \Lambda \in \Gamma_L \backslash L + h \\ Q(\Lambda) \neq 0 \end{subarray}} 
\sgn(\Lambda) \ebf \lp -Q(\Lambda) \tau \rp
\lp J_1(\Lambda)  + J_2(\Lambda)  + J_3(\Lambda)  \rp
\end{equation}
where $J_1(\Lambda), J_2(\Lambda)$ and $J_3(\Lambda)$ are defined by
\begin{align*}
J_1(\Lambda) &:=  
{ \log r(\lambda_0)}
\int^\infty_{-\infty} \sqrt{|Q(\Lambda )|v}  \lp  e^{-\nu} - \sgn(Q(\Lambda)) e^\nu \rp
\ebf\lp \frac{|Q(\Lambda)| \lp e^{\nu} + \sgn(Q(\Lambda)) e^{-\nu}  \rp^2 iv }{2}
\rp d\nu , \\
  J_2(\Lambda) &:=  2
\int^\infty_{-\infty}
\sqrt{|Q(\Lambda)|v} \lp  e^{-\nu} - \sgn(Q(\Lambda)) e^\nu \rp
\ebf\lp \frac{|Q(\Lambda)|  \lp e^{\nu} +  \sgn(Q(\Lambda)) e^{-\nu}  \rp^2 iv }{2} \rp
\nu d \nu, \\
J_3(\Lambda) &:=  
\frac{\sgn(\Lambda) 2 \log \varep \sqrt{v}}{\sqrt{AM}} 
\sum_{n \in \Zb} n \int^\varep_1  \lp \lambda_n  t^{-1} - \lambda'_n t \rp
\ebf\lp \frac{ (\lambda_n t^{-1} + \lambda'_n t)^2  iv }{2AM}
\rp \frac{dt}{t}.
\end{align*}
Note that $J_1$ and $J_2$ comes naturally out of the unfolding process Hecke used, which does not work directly for $J_3$.
Like equation \eqref{eq:integral_id}, we can evaluate the terms $J_1(\Lambda)$ and $J_2(\Lambda)$ as
\begin{center}
\begin{tabular}{|c|c|c|}\hline
  &  $-Q(\Lambda) > 0$ & $-Q(\Lambda) < 0$ \\ \hline
$J_1(\Lambda)$ & ${\sgn(\Lambda) \log r(\lambda_0)} \ebf(-Q(\Lambda) iv)$ & 0\\ \hline
$J_2(\Lambda)$ & 0 & $- \sgn(\Lambda) \ebf(2 Q(\Lambda) iv) \Gamma(0, 4 \pi Q(\Lambda) v)$ \\ \hline 
\end{tabular}
\end{center}
Using the identity $d_t \erf(a t + b t^{-1}) = \frac{2}{\sqrt{\pi}}(at - bt^{-1}) e^{-(at + bt^{-1})^2}$, we obtain
$$
2 \sqrt{\frac{v}{AM}} 
\int^\varep_1
\lp
\lambda'_n t  - \lambda_n t^{-1}
\rp
\ebf\lp
\frac{(\lambda_n t^{-1} + \lambda'_n t)^2  iv }{2AM}
\rp
 \frac{dt}{t} 
=
\erf \lp \sqrt{\frac{\pi v }{AM}}  (\lambda'_{n} t  + \lambda_{n} t^{-1} ) \rp  \; |^\varep_1.
$$
Applying this and the identity $\erf(x) = \sgn(x) - \sgn(x) \erfc(|x|)$ to $J_3(\Lambda)$ gives us
\begin{align*}
&\sgn(\Lambda)  J_3(\Lambda)
=
 { \log \varep}
 \sum_{n \in \Zb} n 
\lp 
\erf \lp \sqrt{\frac{\pi v }{AM}}  (\lambda'_{n-1}  + \lambda_{n-1} ) \rp -
\erf \lp \sqrt{\frac{\pi v }{AM}}  (\lambda'_n  + \lambda_n ) \rp
\rp\\
&= 
 { \log \varep}{} \lp \sgn(\lambda_{0} + \lambda'_{0}) - \sgn(\lambda_1 + \lambda'_1) \rp 
+ {\log \varep}{} \sum_{n \in \Zb} 
\sgn(\lambda_{n} + \lambda'_{n})
 \erfc \lp \sqrt{\frac{\pi v }{AM}}  |\lambda'_{n}  + \lambda_{n} | \rp.
\end{align*}
The first term is $- {\sgn(\Lambda) \log \varep}{}$ if $\lambda_0 = - \lambda'_0$ and zero otherwise. 
After substituting these into equation \eqref{eq:I'h_2}, we obtain equation \eqref{eq:FE_I'}.
\end{proof}

\subsection{Proof of Theorem \ref{thm:Main}.}
\label{subsec:proof}
Suppose $(L, Q) = (L_{\af, M}, Q_{\af, M})$ for some $\af \subset \Oc_D \subset \Oc_F$ and $M \in \Nb$. Define
\begin{equation}
\label{eq:theta}
\theta(\tau, L) := 2I'(\tau, -L) + \log \varep_L \tTheta(\tau, L). 
\end{equation}
By Theorem \ref{thm:tTheta_gen} and Proposition \ref{prop:FE_I'}, $\theta(\tau, L) \in H_{1, \rho_{-L}}(\Gamma)$ and the holomorphic part $\theta^+(\tau, L) := \theta(\tau, L) + \tvartheta^*(\tau, L)$ is given by
$$
\theta^+(\tau, L) = \sum_{h \in L^*/L} \ef_h \sum_{\Lambda \in \Gamma_L \backslash L+ h, \; Q(\Lambda) < 0} \sgn(\Lambda) \log r(\Lambda) q^{-Q(\Lambda)} + \log \varep_L \tTheta^+(\tau, L). 
$$
Theorem \ref{thm:tTheta_gen} implies that $\theta(\tau, L)$ satisfies the Theorem \ref{thm:Main} with $\kappa = \kappa_L$.

Now, we can reduce $\kappa$ as follows. 
Let $\bfrak$ be another ideal of $\Oc_D$ such that $\af = \bfrak \cdot  (\mu \Oc_D)$ with a totally positive element $\mu \in F$ satisfying $\mu - 1 \in M \df_D$.
This is an equivalence relation, under which there are only finitely many equivalence classes of $\Oc_D$-ideals.
Let $A_1 := [\Oc_D : \bfrak]$ and $N_1\in \Nb$ such that $N_1 M = 2A (N'_1)^2$ for some $N'_1 \in \Nb$. 
Then $\vartheta(\tau, L) = \vartheta(\tau, L_{\bfrak})$ and $\theta_1(\tau, L) := \theta(\tau, L_{\bfrak}) + \log \left| \frac{\mu}{\mu'} \right| \vartheta(\tau, -L)$ also satisfies Theorem \ref{thm:Main} with $\kappa = \kappa_{L_{\bfrak}}$.
Suppose $\kappa' = \gcd(\kappa_L, \kappa_{L_\bfrak}) = c_0 \kappa_L + c_1 \kappa_{L_\bfrak}$ with $c_0, c_1 \in \Zb$, then $\frac{c_0\kappa_L \cdot \theta(\tau, L) + c_1 \kappa_{L_\bfrak} \cdot \theta_1(\tau, L)}{\kappa_2}$ will satisfy the statement in Theorem \ref{thm:Main} with $\kappa = \kappa'$. 
Since there are only finitely many equivalence classes, we can repeated this process finitely many times to find the minimal $\kappa$, which only depends on the data $\Oc_D$ and $M$. This proves the last part of Theorem \ref{thm:Main}.

In particular when the $\Oc_D$-ideal $\af$ is proper and $\gcd(A, M) = 1$, we can choose a totally positive $\mu \in F$ such that $\mu - 1 \in M \df_D$ and $\mu \Oc_D = \af \bfrak^{-1}$ for some integral, proper $\Oc_D$-ideal $\bfrak$ relatively prime to $\af$. Then it is possible to reduce $\kappa$ to at least $24 M^3 \phi(2M) $.

\section{Scalar-Valued Result and Numerical Examples.}
In this section, we will use the result in \cite{Stark88} to produce a scalar-valued version of $\vartheta(\tau, L)$, and prove the scalar-valued version of Theorem \ref{thm:Main}.

\subsection{Reducing the level.}
\label{subsec:lev_red}
We first need to reduce the level of certain vector-valued automorphic forms.
Fix a fundamental discriminant $D > 1$, an integral ideal $\mf \subset \Oc_F \subset F := \Qb(\sqrt{D}) \subset \Rb$ and denote $M = \Nm(\mf)$, $N = MD$ and $\chi_D(\cdot) = \lp \frac{D}{\cdot} \rp$ the quadratic Dirichlet character.
Let $\af \subset \Oc_F$ be an arbitrary integral ideal relatively prime to $\mf$ with $A:=\Nm(\af )$, and denote $L = L_{\af \df, M}$. Then $L^* = \af$ and there is a canonical surjection map of finite abelian groups 
\begin{equation}
\label{eq:pi_surj}
 L^*/L = \af/M\df \af \to \af/\mf  \af \stackrel{\cong}{\hookrightarrow} \Oc_F/\mf.
\end{equation}
This induces a natural, linear map from $\Cb[L^*/L]$ to $\Cb[\af/\mf \af]$.
%
%
Let $\{\ef_\sigma: \sigma \in \af / \mf \af  \}$ be the canonical basis of $\Cb[\af / \mf \af ]$. 
Under this and the canonical basis of $\Cb[L^*/L]$, the linear map is given by the matrix
\begin{equation}
  \label{eq:Cc_mf}
  \Cc_{\af, \mf} := (\mathds{1}_{\mf \af}(h - \sigma))_{h \in L^*/L, \sigma \in \af/\mf \af},
\end{equation}
where $\mathds{1}_{\mf \af}$ is the characteristic function of $\mf \af \subset \af$.
%
%
%

Define a representation $\rho_{\mf}$ of $\Gamma_0(N)$ on $\Cb[\Oc_F/\mf]$ by 
\begin{equation}
  \label{eq:rho_mf}
  \rho_{\mf}(\gamma) \ef_\sigma = \chi_D(d) \ef_{d \sigma}, \gamma = \smat{*}{*}{*}{d} \in \Gamma_0(N).
\end{equation}
Since $\af$ and $\mf$ are relatively prime to each other, there is a canonical isomorphism between $\Cb[\af/\mf \af]$ and $\Cb[\Oc_F/\mf]$ induced by $\af/\mf  \af \stackrel{\cong}{\hookrightarrow} \Oc_F/\mf$.
Conjugating $\rho_\mf$ by this isomorphism gives rise to a representation of $\Gamma_0(N)$ on $\Cb[\af/\mf \af]$, which we also denote by $\rho_\mf$.
%
We can now use Theorem 3 in \cite{Stark88} to study precisely the effect of $\Cc_{\af, \mf}$ on the level of the vector-valued automorphic forms such as $ \Theta(N\tau,  L; t)$ in equation \eqref{eq:theta_explicit}. 
\begin{prop}
  \label{prop:mod_N_Theta}
In the notations above, we have $\Cc_{\af, \mf} \cdot \Theta(N\tau, \pm L; t) \in \Ac_{1, \rho_{\mf}}(\Gamma_0(N))$ for any $t \in \Rb^\times_+$.
\end{prop}

%

\begin{proof}
%
For each $\sigma \in \af/\mf \af$, we denote the $\ef_\sigma$-component of $\Cc_{\af, \mf} \Theta(\tau, L; t)$ by $\Theta_\sigma(\tau, \af, \mf; t)$.
  From the definition, it is easy to check that
\begin{equation}
\label{eq:Theta_sigma}
\Theta_\sigma(AN \tau, \af, \mf ; t) = \sqrt{v} \sum_{\lambda \in \af \mf  + \sigma} (\lambda' t + \lambda t^{-1}) \ebf \lp \Nm(\lambda) u + \frac{1}{2} ((\lambda t^{-1})^2 + (\lambda' t)^2) iv \rp.
\end{equation}
Theorem 3 in \cite{Stark88} with the choice of integral ideal $I=\af\mf$ and $\sigma\in \af$ then implies $\Cc_{\af, \mf} \Theta(AN \tau, L; t)) \in \Ac_{1, \rho_{\mf}}(\Gamma_0(AN))$.
It follows readily from equation (\ref{eq:Theta_sigma}) that $\Cc_{\af, \mf} \Theta(N(\tau + 1), L; t) = \Cc_{\af, \mf} \Theta(N\tau , L; t).$  Since $\Gamma_0(N)$ is generated by $\smat{1}{1}{0}{1}$ and the matrices $\smat{a}{b}{c}{d} \in \Gamma_0(N)$ satisfying $A \mid b$, we obtain the desired result. The same argument works for $-L$.
\end{proof}

\begin{rmk}
  \label{rmk:lin_comb}
Integrating over $t,$ it is easy to verify from Proposition \ref{prop:mod_N_Theta} and  equation (\ref{eq:vartheta}) that 
\begin{equation}
  \Cc_{\af, \mf} \cdot \vartheta(N\tau, \pm L) = \lp \sum_{\lambda \in \Gamma_L \backslash (\af\mf + \sigma),\; \pm \Nm(\lambda) > 0 } \sgn(\lambda) q^{|\Nm(\lambda)|/A} \rp_{\sigma \in \af/\mf \af} \in S_{1, \rho_\mf}(\Gamma_0(N)).
\end{equation}
\end{rmk}

Recall from Section \ref{subsec:vvTheta} that $\Theta(\tau, \pm L ; t) \in \Ac_{1, \rho_{\pm L}}(\SL_2(\Zb))$ for all $t \in \Rb^\times_+$. 
It turns out that one can bootstrap a more general result out of the proposition above.

\begin{prop}
  \label{prop:mod_N_gen}
In the notations above, $\Cc_{\af, \mf} \cdot f(N\tau)$ is in $\Ac_{1, \rho_\mf}(\Gamma_0(N))$ for all $f \in \Ac_{1, \rho_{\pm L}}(\SL_2(\Zb))$.
\end{prop}

\begin{proof}
  The proof is analogous to that of Proposition \ref{prop:GammaN}.
For $\gamma = \smat{a}{b}{Nc}{d}$, recall that $\gamma_N =  \smat{a}{Nb}{c}{d}$ as in Proposition \ref{prop:GammaN}. For an arbitrary $f \in \Ac_{1, \rho_{L}}(\SL_2(\Zb))$, we want to show that
$$
(\Cc_{\af, \mf} \cdot f(N\tau)) \mid_{1} \gamma = \rho_\mf(\gamma) \cdot \Cc_{\af, \mf} \cdot f(N\tau).
$$
Since $f(N \tau) \mid_1 \gamma = \rho_L(\gamma_N) \cdot f(N\tau)$, it suffices to prove the identity $M_\gamma \cdot f(N\tau) = 0$, where
$$
M_\gamma := \Cc_{\af, \mf} \cdot \rho_L(\gamma_N) - \rho_\mf(\gamma) \cdot \Cc_{\af, \mf}.
$$
By Proposition \ref{prop:mod_N_Theta}, the equality above holds with $f(\tau) = \Theta(\tau, L; t)$ for any $t \in \Rb^\times_+$.
All Fourier coefficients of the corresponding identity do vanish and therefore, arguing as in Proposition \ref{prop:GammaN}, we know that the right kernel of $M_\gamma$ contains $\{\ef_h - \ef_{-h}: h \in L^*/L\}$. 
From equation \eqref{eq:Weil_rep}, we have $\rho_L(S^2) \ef_h = -\ef_{-h}$ for all $h \in L^*/L$.
Then any $f = \sum_{h \in L^*/L} f_h \ef_h \in \Ac_{1, \rho_L}(\SL_2(\Zb))$ satisfies $f_h = -f_{-h}$. Thus, $ 2 f(N\tau) = \sum_{h \in L^*/L} f_h(N\tau) (\ef_h - \ef_{-h}) $ is in the right kernel of $M_\gamma$. The same argument works for $-L$.
\end{proof}


\subsection{Ray class group character.}
\label{subsec:ray}
Let $\mbf= \mf \cdot \infty_1$ be a modulus with $M = \Nm(\mf)$, $I_\mbf$ be the group of fractional ideals of $\Oc_F$ relatively prime to $\mf$ and
 $P_{\mbf}$ be the group of principal fractional ideals generated by elements $\mu \in F$ such that $\mu \equiv^\times 1 \bmod{\mf}$ and $\mu > 0$.
Consider a ray class group character of conductor $\mbf$
\begin{equation}
  \label{eq:varphi}
  \varphi: \Cl_\mbf := I_\mbf / P_{\mbf} \to \Cb^\times.
\end{equation}
The class group $\Cl_F$ of $F$ is a natural quotient of $\Cl_\mbf$. 
%
The presence of $\infty_1$ in $\mbf$ is equivalent to
\begin{equation}
  \label{eq:varphif}
\varphi((\mu)) = \varphif(\mu) \cdot \sgn(\mu).  
\end{equation}
with $\varphif$ a character on $(\Oc_F/\mf)^\times$ satisfying $\varphif(\varep) = \sgn(\varep)$ for all $\varep \in \Oc_F^\times$. 
After extending by zero, we can view $\varphif$, resp.\ $\varphi$, as a map on $F$, resp.\ fractional ideals of $\Oc_F$.
By a slight abuse of notation, we write $\varphi(\lambda) := \varphi((\lambda))$.

Let $N = DM$ and $\chi$ be the product of $\chi_D$ and the restriction of $\varphi$ to $\Qb$.
Associated to $\varphi$ is a Hecke eigenform $f_\varphi \in S_{1, \chi}(\Gamma_0(N))$ given by
\begin{equation}
  \label{eq:fvarphi}
f_\varphi(\tau) := \sum_{\bfrak \subset \Oc_F} \varphi(\bfrak) q^{\Nm(\bfrak)}
=
\sum_{[\af] \in \Cl_F} \varphi(\af^{-1}) f_{\varphi, \af}(\tau),\;
f_{\varphi, \af}(\tau) := \sum_{(\lambda) \subset \af} \varphi(\lambda ) q^{\Nm((\lambda)\af^{-1} )}.
\end{equation}
Even though $\varphi(\af)$ depends on the representative of $[\af] \in \Cl_F$, the product $\overline{\varphi}(\af) f_{\varphi, \af}$ is independent of such choice.

For fixed integral ideal $\af \in I_\mbf$, let $A = \Nm(\af)$ and $L = L_{\af \df, M}$.
Then we can write $f_{\varphi, \af}(\tau) = {f_{\varphi, \af, +}(\tau) + f_{\varphi, \af, -}(\tau)}$, where
\begin{equation}
\label{eq:fafpm}
f_{\varphi, \af, \pm}(\tau) := \frac{1}{[\Oc^\times_F: \Gamma_L]} \sum_{\lambda \in \Gamma_L \backslash \af, \; \pm \Nm(\lambda) > 0} \varphif(\lambda ) \sgn(\lambda) q^{|\Nm(\lambda)|/A}.
\end{equation}
Now, we can express $f_{\varphi, \af}$ as a linear combination of $\vartheta_{\sigma, \pm}$ as follows.

\begin{prop}
\label{prop:lin_comb}
Let $\Cc_{\varphi} := (\varphif(\sigma))_{\sigma \in \Oc_F/\mf}$ be a row vector.
Then left multiplication by $\Cc_{\varphi}$ is a linear map from $\Ac_{1, \rho_{\mf}}(\Gamma_0(N))$ to $\Ac_{1, \chi}(\Gamma_0(N))$.
In particular, 
\begin{equation}
  \label{eq:lin_comb1}
f_{\varphi, \af, \pm}(\tau) = 
\Cc_{\varphi} \cdot 
\frac{\Cc_{\af, \mf} \cdot \vartheta(N \tau, \pm L) }{[\Oc^\times_F: \Gamma_L]} 
\in S_{1, \chi}(\Gamma_0(N)).
\end{equation}
\end{prop}

\begin{proof}
This is a consequence of Remark \ref{rmk:lin_comb} above and Section 5 of \cite{Stark88}.
\end{proof}

The scalar-valued version of Theorem \ref{thm:Main} is as follows. 
\begin{thm}
  \label{thm:Main_sc}
For each class in $\Cl_F$, fix a representative $\af \in I_\mbf$ with $A = \Nm(\af)$. 
Let $\varphi$ and $f_\varphi \in S_{1, \chi}(\Gamma_0(N))$ be as above. 
There exists a harmonic Maass form $\tf_\varphi \in H_{1, \overline{\chi}}(\Gamma_0(N))$ such that $\xi_1 \tf_\varphi = f_\varphi$ and the holomorphic part of $\tf_\varphi$ has the Fourier expansion $\sum_{n \gg -\infty} c^+_\varphi(n) q^n$ satisfying 
\begin{equation}
  \label{eq:c+}
  c^+_\varphi(n) - 
\sum_{[\af] \in \Cl_F} \varphi(\af) \sum_{(\lambda) \subset \af, \; \Nm((\lambda) \af^{-1}) = n } \overline{\varphi}(\lambda)  \log \left| \frac{\lambda}{\lambda'} \right|
\in 
\frac{1}{\kappa_\mf} \Zb[\varphi] \cdot \log \varep_F
\end{equation}
with $\kappa_\mf \mid 48 M^3 \phi(2M)$.
\end{thm}

\begin{proof}
  For each representative $\af \in I_\mbf$, let $L = L_{\af\df, M}$ and $\tvartheta(\tau, \pm L) \in H_{1, \rho_{\pm L}}(\SL_2(\Zb))$ be the harmonic Maass form constructed in Theorem \ref{thm:Main}.
Consider 
\begin{equation}
  \label{eq:tfaf}
  \tf_{\varphi, \af}(\tau) := 
\frac{1}{[\Oc^\times_F: \Gamma_L]} 
\Cc_{\overline{\varphi}} \cdot \lp \Cc_{\af, \mf} \cdot \tvartheta(N \tau, L) 
+
 \Cc_{\af, \mf} \cdot \tvartheta(N \tau, - L) \rp.
\end{equation}
By Propositions \ref{prop:mod_N_gen} and \ref{prop:lin_comb}, $\tf_{\varphi, \af} \in H_{1, \overline{\chi}}(\Gamma_0(N))$ and $\xi_1 \tf_{\varphi, \af} = f_{\varphi, \af}$. Its holomorphic part $\tf^+_{\varphi, \af}$ has the Fourier expansion 
$$
\tf^+_{\varphi, \af}(\tau) = \sum_{n \gg 0} c^+_{\varphi, \af}(n) q^n.
$$
As a consequence of Theorem \ref{thm:Main}, the coefficient $c^+_{\varphi, \af}(n)$ satisfies
$$
c^+_{\varphi, \af}(n) -
\sum_{(\lambda) \subset \af, \; {\Nm((\lambda)\af^{-1})} = n }  \overline{\varphi}((\lambda)) \log \left| \frac{\lambda}{\lambda'} \right|
\in 
\frac{1}{\kappa \cdot [\Oc_F^\times: \Gamma_L]} \Zb[\varphi] \cdot \log \varep_L.
$$
Since $\Oc_F^\times = \{\pm \varepsilon_F^n: n \in \Zb\}$, we have $\frac{\log \varep_L}{[\Oc^\times_F: \Gamma_L]} = \half \log \varepsilon_F$.
By Theorem \ref{thm:Main}, we can choose $\kappa = 24 M^3 \phi(2M)$.
Summing over the representatives of all the classes in $\Cl_F$ with respect to $\overline{\varphi}(\af^{-1}) = \varphi(\af)$ finishes the proof.
\end{proof}

When $F$ has class number one, we can choose $\af = \Oc_F$ as the only summand in the sum over $\Cl_F$ and arrange to have
\begin{align*}
  \sum_{\begin{subarray}{c} (\lambda) \subset \Oc_F \\ \Nm((\lambda)) = n \end{subarray} } \overline{\varphi}(\lambda)  \log \left| \frac{\lambda}{\lambda'} \right| &= 
\frac{1}{2} \lp   \sum_{\begin{subarray}{c} (\lambda) \subset \Oc_F \\ \Nm((\lambda)) = n \end{subarray}}
\overline{\varphi}(\lambda)  \log \left| \frac{\lambda}{\lambda'} \right|  + 
  \sum_{\begin{subarray}{c} (\lambda) \subset \Oc_F \\ \Nm((\lambda)) = n \end{subarray}} \overline{\varphi}(\lambda')  \log \left| \frac{\lambda'}{\lambda} \right| \rp  \\
&= \cbf_{\varphi}(n) \in \Cb/(\Zb[\varphi] \log \varep_F).
\end{align*}
Theorem \ref{thm:Main_sc_1} then follows from Theorem \ref{thm:Main_sc}.

To compare this result with the $p$-adic version in \cite{DLR15a}, we first define the character $\varphi_\heartsuit := \varphi/\varphi'$ with $\varphi'$ the composition of $\varphi$ and the conjugation on $F$. 
The kernel of $\varphi_\heartsuit$ fixes a number field $H$, which is a class field of $F$.
When $\ell$ is an inert prime in $F/\Qb$, Darmon, Lauder and Rotger defined in \cite{DLR15a} an element $u(\varphi_\heartsuit, \ell) \in \Zb[\varphi] \otimes \Oc_H[1/\ell]^\times$ and showed that its $p$-adic logarithm is the $\ell$\tth Fourier coefficient of a generalized overconvergent eigenform of weight one.

When $\ell = \lambda \lambda'$ is a split prime in $F/\Qb$ with $\sigma_\lambda, \sigma_{\lambda'} \in \Gal(H/F)$ the respective Frobenius elements, the $\ell$\tth Fourier coefficient in the $p$-adic setting is zero, whereas we can define an element $u(\varphi_\heartsuit, \ell) \in \Zb[\varphi]\otimes (F^\times/\Oc_F^\times)$ by
\begin{equation}
  \label{eq:ul}
  u(\varphi_\heartsuit, \ell) := \sum_{\sigma \in \Gal(F/\Qb)} \varphi^{-1}_\heartsuit(\sigma \sigma_\lambda \sigma^{-1}) \otimes \sigma(\lambda) \in \Zb[\varphi] \otimes (\Oc_F[1/\ell]^\times / \Oc_F^\times).
\end{equation}
After extending the complex logarithm by $\Zb[\varphi]$-linearity, we have
\begin{equation}
  \label{eq:main_id}
  \cbf_\varphi(\ell) = \log | u(\varphi_\heartsuit, \ell) | \in \Cb/(\Zb[\varphi] \log \varep_F).
\end{equation}
This might be a complicated way to write $\cbf_\varphi(\ell)$ since there are only two elements in $\Gal(F/\Qb)$. Nevertheless, it shows that the archimedean and non-archimedean situations complement each other.
In the latter, the element $u(\varphi_\heartsuit, \ell)$ is well-defined in $\Zb[\varphi] \otimes \Oc_H[1/\ell]^\times$ and can be used to generate class fields of $F$.
In the former, the $\ell$-unit $u(\varphi_\heartsuit, \ell)$ lies in the ground field $F$ and is only well-defined up to $\Zb[\varphi] \log \varep_F$. 
This reflects the difficulty in choosing a canonical harmonic Maass form with a given $\xi$-image. 

\subsection{Examples.}
The first example was studied in detail in Section 7 of \cite{Li16}, where $F = \Qb(\sqrt{29})$, $\mf = \lp \frac{3 + \sqrt{29}}{2} \rp$ and $\varphi(2) = i$. The cusp form $f_\varphi$ is the unique normalized eigenform in $S_{1, \chi}(\Gamma_0(145))$ with $\chi = \lp \frac{\cdot}{29} \rp \cdot \varphi\mid_\Qb$. 
The second author showed loc.\ cit.\ that there exists a unique harmonic Maass form $\mathcal{F}_\varphi$ in $H_{1, \overline{\chi}}(\Gamma_0(145))$ that maps to $f_\varphi$ and has only a simple pole at $i\infty$.  
Using a modularity-based algorithm, the second author numerically calculated the Fourier coefficients $c^+_\varphi(n)$ of the holomorphic part of $\mathcal{F}_\varphi$ as complex numbers. Then they were identified as $\Zb[\varphi]$-linear combinations of logarithm of numbers in $F^\times$. 
When $n = \ell = \lambda \lambda'$ is a split prime in $F$, there was a precise conjecture about the shape of $c^+_\varphi(\ell)$, which is implied by Theorem \ref{thm:Main_sc_1} if $\kappa_\mf = 1$.

The second example is about a rather classical weight one cusp form studied by Hecke \cite{Hecke26}. Set $D = 12, F = \Qb(\sqrt{D}), L = \af = \df_F = 2\sqrt{3} \Oc_F$, $M = 1$, and $L^* = \Oc_F$. 
The discriminant kernel $\Gamma_L$ is generated by $\varep_L = 7 + 4\sqrt{3}$, which is the square of the fundamental unit $\varep_F = 2 + \sqrt{3}$. 
In this case, the theta series $\vartheta(\tau, L)$ is a 12-dimensional, holomorphic vector-valued weight one cusp form on $\SL_2(\Zb)$. The components correspond to $L^*/L = \Oc_F/\df_F$. 
It turns out 8 of the 12 components vanish identically. The other 4 components correspond to $h = \pm 1, \pm ( 2 + \sqrt{3}) \in \Oc_F/\df_F$ and satisfy $\vartheta_{1} = -\vartheta_{-1} = \pm \vartheta_{\pm (2 + \sqrt{3})}$. 
So $\ef := \ef_1 + \ef_{2 + \sqrt{3}} - \ef_{-1} - \ef_{-(2+\sqrt{3})}$ is an eigenvector of $\rho_L(T)$ and $\rho_L(S)$ with eigenvalues $\ebf(1/12)$ and $-i$ respectively. 
%
This implies 
$$
\vartheta_1 \lp -\frac{1}{\tau}, L \rp = - i \tau \vartheta_1(\tau, L), \;
\vartheta_1 \lp \tau + 1, L \rp = \ebf \lp \frac{1}{12} \rp \vartheta_1 (\tau, L).
$$
Since $\eta(\tau)^2 := q^{1/12} \prod_{n = 1}^\infty (1 - q^n)^2$ also satisfies this transformation property and has zero only at the cusps, they are equal up to a multiplicative constant. 
By comparing the first non-vanishing Fourier coefficient, Hecke obtained
$$
\vartheta_1(\tau, L) = \sum_{\begin{subarray}{c} \lambda \in \Gamma_L \backslash \Oc_F\\ \lambda \equiv 1 \bmod 2\sqrt{3}\\ \Nm(\lambda) > 0  \end{subarray}}
\sgn(\lambda) q^{\Nm(\lambda)/12} = \eta^2(\tau) \in S_{1, \rho^{\mathrm{sc}}_L}(\Gamma),
$$
with $\Gamma = \SL_2(\Zb)$ and $\rho^{\mathrm{sc}}_L$ the restriction of $\rho_L$ on the eigenspace spanned by $\ef$.

The representation $\rho^{\mathrm{sc}}_L$ is a multiplier system of $\Gamma$ and can be expressed in terms of Dedekind sums. 
Similarly, $\rho^{\mathrm{sc}}_{-L}$ is the restriction of $\rho_{-L}$ on $\ef$, and the conjugate of $\rho^{\mathrm{sc}}_L$. 
The space $S_{1, \rho^{\mathrm{sc}}_L}(\Gamma)$ is spanned by $\vartheta_1(\tau, L)$ since the 12\tth power of any form in that space is in the one dimensional space $S_1(\Gamma)$. 
On the other hand, the space $S_{1, \rho^{\mathrm{sc}}_{-L}}(\Gamma)$ is trivial and $\vartheta(\tau, -L)$ vanishes identically. 
Therefore the elements $\lambda \in \Oc_F$ with negative norm will never contribute to any holomorphic modular object. 
It is worth noting that $\eta^2(\tau)$ can also be constructed from lattices of signature $(2, 0)$, or equivalently from ramified characters associated to the imaginary quadratic fields $\Qb(\sqrt{-1})$ and $\Qb(\sqrt{-3})$ (see e.g.\ \cite{Hecke26, Sch53}). 
So the methods of \cite{DL15} and \cite{Ehlen_thesis} could also be used to produce a $\xi_1$ preimage of $\eta^2(\tau)$.

Before constructing the harmonic Maass form $\tvartheta(\tau, L)$, notice that $\vartheta(\tau, L) = \vartheta(\tau, L_{\df_F/2, 1})$. So we can suppose that $L = L_{\df/2, 1}$. 
First, we need to construct $\tTheta(\tau, L)$. This can be done as in Section \ref{subsec:general} with $N = 6$ in equation \eqref{eq:N}. Then $|L^*/6L| = 432 = 6^2 |L^*/L|$. 
Using the procedures in Section \ref{subsec:special}, we can construct $\tTheta(\tau, 6L)$. Its holomorphic part $\tTheta^+(\tau, 6L)$ has rational Fourier coefficients with denominators bounded by 6. 
Then equation \eqref{eq:tTheta_gen} defines $\tTheta(\tau, L)$, whose holomorphic part is given by equation \eqref{eq:tTheta+}. There, the sum over $\Gamma_0(6) \backslash \Gamma$ has 12 summands, each of which is the product of a $12 \times 432$ matrix $\frac{N_\gamma}{N}
\rho_{-L}^{-1}(\gamma) \cdot  \Cc_{L, N} \cdot \rho_{-NL}(\gamma_N) $ and a vector $ \tTheta^+(\tau_\gamma, NL)$ of size 432. 
By Theorem \ref{thm:tTheta_gen}, all the components of $\tTheta^+(\tau, L)$ have rational Fourier coefficients with bounded denominator.
Using SAGE \cite{SAGE}, we have numerically implemented this procedure and calculated the Fourier coefficients of $\tTheta^+_h(\tau, L)$ for each $h \in L^*/L$. 
Since $\rho_{-L}(\smat{-1}{0}{0}{-1})\ef_h = \ef_{-h}$ and the weight is odd, we have $\tTheta_h(\tau, L) = - \tTheta_{-h}(\tau, L)$ for all $h \in L^*/L$.
When $h =  0, \sqrt{3}, 3$ and $3 + \sqrt{3}$, $h = -h$ and $\tTheta_h(\tau, L) = 0$.
For the other $h$, the Fourier expansions are listed in the table below.
\begin{center}
  \begin{tabular}{|c||c|}
\hline
   $h$ & $3 \cdot [\Gamma: \Gamma_0(6)] \cdot \tTheta^+_h(\tau, L)$ \\ \hline
$1$ & ${4}{} q^{-1/12} - {584}{} q^{11/12} - {9764}{} q^{23/12} - {88024}{} q^{35/12} + O(q^{47/12})$  \\ \hline
$1 + \sqrt{3}$ & ${8}{} q^{2/12} - {1184}{} q^{14/12} - {17152}{} q^{26/12} - {142912}{} q^{38/12} + O(q^{50/12})$ \\ \hline
$2$ & ${}{} q^{-4/12} + 192 q^{8/12} + {4736}{}q^{20/12} + {51052}{} q^{32/12} + 365634 q^{44/12} + O(q^{56/12})$ \\ \hline
$2 + \sqrt{3}$ & ${2}{}q^{-1/12} + {380}{} q^{11/12} + {8714}{} q^{23/12} + {85060}{} q^{35/12} + O(q^{47/12})$\\ \hline
  \end{tabular}
\end{center}

Now as in \eqref{eq:theta}, we define $\tvartheta(\tau, L) := 2 I'(\tau, -L) + \log \varep_L \cdot \tTheta(\tau, L) \in H_{1, \rho_{-L}}(\Gamma)$, which satisfies $\xi_1 \tvartheta(\tau, L) = \vartheta(\tau, L)$. 
Then the eigen-component 
\begin{equation}
  \label{eq:eigen_comp}
  \tilde{f}(\tau) := \frac{1}{4} 
\lp \tvartheta_1(\tau, L) + \tvartheta_{2+\sqrt{3}}(\tau, L) - \tvartheta_{-1}(\tau, L) - \tvartheta_{-(2+\sqrt{3})}(\tau, L) \rp \in H_{1, \rho^{\mathrm{sc}}_{-L}}(\Gamma)
\end{equation}
maps to $\eta^2(\tau)$ under $\xi_1$. 
Note that $\tilde{f} = \frac{\tvartheta_1(\tau, L) + \tvartheta_{2 + \sqrt{3}}(\tau, L) }{2}$.
Its holomorphic part $\tilde{f}^+$ can be written explicitly as
$$
\tilde{f}^+(\tau) = 
\frac{1}{2}
\sum_{h \in \{1, 2 + \sqrt{3}\}  }
\lp  \tTheta^+_h(\tau, L) \cdot \log \varepsilon_L  + \sum_{\begin{subarray}{c}\Lambda \in \Gamma_L \backslash L + h \\ Q(\Lambda) < 0\end{subarray}} a(\Lambda) q^{|Q(\Lambda)|} \rp
$$
using Prop.\ \ref{prop:FE_I'}. Suppose $\tilde{f}^+$ has the Fourier expansion $\sum_{n \ge -1, n \equiv 11\bmod{12}} c^+(n) q^{n/12}$, then the coefficients have the shape
\begin{equation}
\label{eq:c+}
c^+(n) = - \frac{1}{12} \log \left| \frac{u(n)}{u(n)'} \right|,
\end{equation}
where $u(n) \in \Oc_F$ is a unit outside primes dividing $n$.
For example, $c^+(-1) = \frac{2 + 4 }{2 \cdot 3 \cdot [\Gamma: \Gamma_0(6)]} \log \varepsilon_L = \frac{\log \varepsilon_F}{6}$ and $u(-1) = \varepsilon_F^{-1}$.
The next coefficient $c^+(11)$ has two parts. One contribution comes from $\tTheta^+_h$, which is $\frac{-584 + 380}{2 \cdot 3 \cdot [\Gamma: \Gamma_0(6)]} \log \varepsilon_L = -\frac{17}{6} \log ( \varepsilon_F / \varepsilon_F')$. The other contribution is
\begin{align*}
\frac{1}{2} \sum_{\begin{subarray}{c} h \in \{1,  2 + \sqrt{3} \} \\ \Lambda \in \Gamma_L \backslash L + h\\  Q(\Lambda) = -11 \end{subarray}} 
a(\Lambda) 
&= \frac{1}{2} \lp 
\log \left| \frac{1 + 2\sqrt{3}}{1-2\sqrt{3}} \right| 
- \log \left| \frac{-17 - 10\sqrt{3}}{-17 + 10\sqrt{3}} \right| 
- \log \left| \frac{-4 - 3\sqrt{3}}{-4 + 3\sqrt{3}} \right| 
+ \log \left| \frac{8 + 5\sqrt{3}}{8 - 5\sqrt{3}} \right| 
\rp\\
&= 2 \log\left| \frac{1 + 2\sqrt{3}}{1-2\sqrt{3}} \right| - \log \left| \frac{\varepsilon_F}{\varepsilon_F'} \right| .
\end{align*}
Therefore, $c^+(11) = 2 \log |(1 + 2\sqrt{3})/(1-2\sqrt{3})| - 23/6 \log (\varepsilon_F/ \varepsilon'_F) $ and $u(11) = (1 - 2\sqrt{3})^{24} \cdot \varepsilon_F^{46}$.
Some of the other $u(n)$ with $n \le 300$ are listed below.
\begin{center}
  \begin{tabular}{|c||c |c| c| c| c|} \hline
$n$ & 23 & 35   & 59 & 95 & 275   \\ \hline
$u(n)$ &  $(2 - 3\sqrt{3})^{24} \varepsilon_F^{187}$ & $\varepsilon_F^{494}$  & $(4 + 5\sqrt{3})^{24} \varepsilon_F^{2748}$ & $\varepsilon_F^{21607}$ & $(1 - 2\sqrt{3})^{24} \varepsilon_F^{21099946}$ \\ \hline
$c^+(n)$ & $-39.42199..$ & $-108.42953..$ & $-605.1655..$ & $-4742.58..$ & $-4631291.273..$ \\ \hline
  \end{tabular}
\end{center}
It is interesting to notice that all Fourier coefficients with positive indices are negative. This can probably be shown from the construction of $\tTheta^+_h(\tau, L)$. 
Finally, Stokes' theorem tells us that the Petersson norm of $\eta^2(\tau)$ is $c^+(-1) =\frac{\log(\varepsilon_F)}{6}$.

\bibliographystyle{amsplain}
\bibliography{bib_HeckeCusp}{}

\providecommand{\bysame}{\leavevmode\hbox to3em{\hrulefill}\thinspace}
\providecommand{\MR}{\relax\ifhmode\unskip\space\fi MR }
\providecommand{\MRhref}[2]{%
  \href{http://www.ams.org/mathscinet-getitem?mr=#1}{#2}
}
\providecommand{\href}[2]{#2}
\begin{thebibliography}{10}

\bibitem{AGHM2}
Fabrizio Andreatta, Eyal~Z. Goren, Benjamin Howard, and Keerthi~Madapusi Pera,
  \emph{Faltings heights of abelian varieties with complex multiplication},
  (2015).

\bibitem{Borcherds98}
Richard~E. Borcherds, \emph{Automorphic forms with singularities on
  {G}rassmannians}, Invent. Math. \textbf{132} (1998), no.~3, 491--562.

\bibitem{Bruinier99}
Jan~H. Bruinier, \emph{Borcherds products and {C}hern classes of
  {H}irzebruch-{Z}agier divisors}, Invent. Math. \textbf{138} (1999), no.~1,
  51--83.

\bibitem{Bruinier02}
\bysame, \emph{Borcherds products on {O}(2, {$l$}) and {C}hern classes of
  {H}eegner divisors}, Lecture Notes in Mathematics, vol. 1780,
  Springer-Verlag, Berlin, 2002.

\bibitem{BF04}
Jan~H. Bruinier and Jens Funke, \emph{On two geometric theta lifts}, Duke Math.
  J. \textbf{125} (2004), no.~1, 45--90.

\bibitem{BO10}
Jan~H. Bruinier and Ken Ono, \emph{Heegner divisors, {$L$}-functions and
  harmonic weak {M}aass forms}, Ann. of Math. (2) \textbf{172} (2010), no.~3,
  2135--2181.

\bibitem{BY09}
Jan~H. Bruinier and Tonghai Yang, \emph{Faltings heights of {CM} cycles and
  derivatives of {$L$}-functions}, Invent. Math. \textbf{177} (2009), no.~3,
  631--681.

\bibitem{Buell89}
Duncan~A. Buell, \emph{Binary quadratic forms}, Springer-Verlag, New York,
  1989, Classical theory and modern computations.

\bibitem{DMZ12}
Atish Dabholkar, Sameer Murthy, and Don Zagier, \emph{Quantum black holes, wall
  crossing, and mock modular forms},  (2012).

\bibitem{DLR15a}
Henri Darmon, Alan Lauder, and Victor Rotger, \emph{Overconvergent generalised
  eigenforms of weight one and class fields of real quadratic fields}, Adv.
  Math. \textbf{283} (2015), 130--142.

\bibitem{DS74}
Pierre Deligne and Jean-Pierre Serre, \emph{Formes modulaires de poids {$1$}},
  Ann. Sci. \'Ecole Norm. Sup. (4) \textbf{7} (1974), 507--530 (1975).

\bibitem{SAGE}
The~Sage Developers, \emph{{S}agemath, the {S}age {M}athematics {S}oftware
  {S}ystem ({V}ersion 6.9)}, 2015, {\tt http://www.sagemath.org}.

\bibitem{DIT11}
William Duke, {\"O}zlem. Imamo{\=g}lu, and {\'A}rp{\'a}d. T{\'o}th, \emph{Cycle
  integrals of the {$j$}-function and mock modular forms}, Ann. of Math. (2)
  \textbf{173} (2011), no.~2, 947--981.

\bibitem{DL15}
William Duke and Yingkun Li, \emph{Harmonic {M}aass forms of weight 1}, Duke
  Math. J. \textbf{164} (2015), no.~1, 39--113.

\bibitem{EOT11}
Tohru Eguchi, Hirosi Ooguri, and Yuji Tachikawa, \emph{Notes on the {$K3$}
  surface and the {M}athieu group {$M_{24}$}}, Exp. Math. \textbf{20} (2011),
  no.~1, 91--96.

\bibitem{Ehlen_thesis}
Stephan Ehlen, \emph{{CM} values of regularized theta lifts}, 2013, Thesis
  (Ph.D.)--T.U. Darmstadt.

\bibitem{GZ85}
Benedict~H. Gross and Don~B. Zagier, \emph{On singular moduli}, J. Reine Angew.
  Math. \textbf{355} (1985), 191--220.

\bibitem{Hecke26}
Erich Hecke, \emph{Zur {T}heorie der elliptischen {M}odulfunktionen (={W}erke,
  no. 23)}, Math. Ann. \textbf{97} (1926), no.~1, 210--242.

\bibitem{HZ76}
Friedrich Hirzebruch and Don~B. Zagier, \emph{Intersection numbers of curves on
  {H}ilbert modular surfaces and modular forms of {N}ebentypus}, Invent. Math.
  \textbf{36} (1976), 57--113.

\bibitem{Kudla16}
Stephen Kudla, \emph{Another product for a {B}orcherds form}, Advances in the
  theory of automorphic forms and their {$L$}-functions, Contemp. Math., vol.
  664, Amer. Math. Soc., Providence, RI, 2016, pp.~261--294.

\bibitem{Kudla81}
Stephen~S. Kudla, \emph{Holomorphic {S}iegel modular forms associated to {${\rm
  SO}(n,\,1)$}}, Math. Ann. \textbf{256} (1981), no.~4, 517--534.

\bibitem{KM90}
Stephen~S. Kudla and John~J. Millson, \emph{Intersection numbers of cycles on
  locally symmetric spaces and {F}ourier coefficients of holomorphic modular
  forms in several complex variables}, Inst. Hautes \'Etudes Sci. Publ. Math.
  (1990), no.~71, 121--172.

\bibitem{KRY99}
Stephen~S. Kudla, Michael Rapoport, and Tonghai Yang, \emph{On the derivative
  of an {E}isenstein series of weight one}, Internat. Math. Res. Notices
  (1999), no.~7, 347--385.

\bibitem{Li16}
Yingkun Li, \emph{Real-dihedral harmonic {M}aass forms and {CM}-values of
  {H}ilbert modular functions}, Compos. Math. \textbf{152} (2016), no.~6,
  1159--1197.

\bibitem{Li17a}
Yingkun Li, \emph{Harmonic {E}isenstein {S}eries of {W}eight {O}ne}, 2017.

\bibitem{Li17b}
Yingkun Li, \emph{Petersson norm of cusp form associated to real quadratic
  field}, 2017.

\bibitem{Mc03}
William~J. McGraw, \emph{The rationality of vector valued modular forms
  associated with the {W}eil representation}, Math. Ann. \textbf{326} (2003),
  no.~1, 105--122.

\bibitem{Sch09}
Nils~R. Scheithauer, \emph{The {W}eil representation of {${\rm SL}_2(\Bbb Z)$}
  and some applications}, Int. Math. Res. Not. (2009), no.~8, 1488--1545.

\bibitem{Sch53}
Bruno Schoeneberg, \emph{\"{U}ber den {Z}usammenhang der {E}isensteinschen
  {R}eihen und {T}hetareihen mit der {D}iskriminante der elliptischen
  {F}unktionen}, Math. Ann. \textbf{126} (1953), 177--184.

\bibitem{Stark77}
Harold~M. Stark, \emph{Class fields and modular forms of weight one}, Modular
  functions of one variable, {V} ({P}roc. {S}econd {I}nternat. {C}onf., {U}niv.
  {B}onn, {B}onn, 1976), Springer, Berlin, 1977, pp.~277--287. Lecture Notes in
  Math., Vol. 601.

\bibitem{Stark88}
\bysame, \emph{On modular forms of weight one from real quadratic fields and
  theta functions}, J. Ramanujan Math. Soc. \textbf{3} (1988), no.~1, 63--79.

\bibitem{ZwThesis}
Sander~P. Zwegers, \emph{Mock theta functions}, Proefschrift Universiteit
  Utrecht, 2002, Thesis (Ph.D.)--Universiteit Utrecht.

\end{thebibliography}

\end{document}